\documentclass{elsarticle}
\usepackage{graphicx}
\usepackage{amscd,amsmath,amstext,amsfonts,amsbsy,amssymb,amsthm,eufrak}
\usepackage{hyperref}
\usepackage[toc,page]{appendix}
\usepackage{chngcntr}

\newtheorem{theo}{Theorem}
\newtheorem{lemm}[theo]{Lemma}

\newtheorem{prop}[theo]{Proposition}

\newdefinition{defi}{Definition}

\newdefinition{rema}{Remark}

\makeatletter
\def\ps@pprintTitle{
 \let\@oddhead\@empty
 \let\@evenhead\@empty
 \def\@oddfoot{}%
 \let\@evenfoot\let\@oddfoot }
\makeatother

\begin{document}

\title{\large {\textbf{On the blow-up results for a class of strongly  perturbed semilinear heat equations}}}
\author{V. T. Nguyen}
\ead{vtnguyen@math.univ-paris13.fr}
\address{Universit\'e Paris 13, Sorbonne Paris Cit\'e,\\ LAGA, CNRS (UMR 7539), F-93430, Villetaneuse, France.}

\begin{abstract}
We consider in this work some class of strongly perturbed for the semilinear heat equation with Sobolev sub-critical power nonlinearity. We first derive a Lyapunov functional in similarity variables and then use it to derive the blow-up rate. We also classify all possible asymptotic behaviors of the solution when it approaches to singularity. Finally, we describe precisely the blow-up profiles corresponding to these behaviors.
\end{abstract}

\maketitle

\section{Introduction}
We are interested in the following nonlinear parabolic equation: 
\begin{equation}\label{equ:problem}
\left\{
\begin{array}{rcl}
u_t &=& \Delta u + |u|^{p-1}u + h(u), \\
u(0) &=& u_0 \in L^\infty(\mathbb{R}^n),
\end{array}
\right.
\end{equation}
where $u$ is defined for $(x,t) \in \mathbb{R}^n \times [0,T)$, $p$ is a sub-critical nonlinearity,
\begin{equation}\label{equ:rangep}
1 < p, \quad (n - 2)p < n+2.
\end{equation}
The function $h$ is in $\mathcal{C}^1(\mathbb{R}, \mathbb{R})$ satisfying
\begin{equation}\label{equ:h}
j = 0,1,\;|h^{(j)}(z)| \leq M\left( \dfrac{|z|^{p-j}}{\log^{a}(2 + z^2)} + 1\right), \quad |h''(z)| \leq M\dfrac{|z|^{p-2}}{\log^{a}(2 + z^2)},
\end{equation}
where $a > 1$, $M > 0$. Typically, $h(z) = \frac{\mu |z|^{p-1}z}{\log^a(2 + z^2)}$ with $\mu \in \mathbb{R}$.\\

\noindent By standard results, the problem \eqref{equ:problem} has a unique classical solution $u(x,t)$ in $L^\infty(\mathbb{R}^n)$, which exists at least for small times. The solution $u(x,t)$ may develop singularities in some finite time. We say that a function $u: \mathbb{R}^n \times [0,T) \mapsto \mathbb{R}$ is a solution of \eqref{equ:problem} if $u$ solves \eqref{equ:problem} and satisfies 
\begin{equation}\label{equ:condForU}
u, u_t, \nabla u,\nabla^2u \;\text{are bounded and continuous on}\; \mathbb{R}^n \times [0,\tau],\; \forall \tau < T.
\end{equation}
It is said that $u(x,t)$ blows up in a finite time $T < +\infty$ if $u(x,t)$ satisfies \eqref{equ:problem}, \eqref{equ:condForU} and 
$$\lim_{t \to T} \|u(t)\|_{L^\infty(\mathbb{R}^n)} = +\infty.$$
Here we call $T$ the blow-up time of $u(x,t)$. In such a blow-up case, a point $x_0 \in \mathbb{R}^n$ is called a blow-up point of $u(x,t)$ if and only if there exist $(x_n, t_n) \to (x_0, T)$ such that $|u(x_n,t_n)| \to +\infty$ as $n \to +\infty$.\\

Consider $v$ a positive blow-up solution of the associated ODE of \eqref{equ:problem}. It is clear that $v$ is given by 
\begin{equation}\label{equ:vode11}
v' = v^p + h(v), \quad v(T) = +\infty, \quad \text{for some $T > 0$.}
\end{equation}
Since the blow-up solution of \eqref{equ:vode11} satisfies (see Lemma \ref{ap:lemmA1})
\begin{equation}\label{behaV1}
v(t) \sim \kappa(T - t)^{-\frac{1}{p-1}} \quad \text{as $t \to T$, where $\kappa = (p-1)^{-\frac{1}{p-1}}$},
\end{equation}
it is natural to ask whether the blow-up solution $u(t)$ of \eqref{equ:problem} has the same blow-up rate as $v(t)$ does. More precisely, are there constants $c, C > 0$ such that
\begin{equation}\label{equ:blrateIn}
c(T -t)^{-\frac{1}{p-1}} \leq \|u(t)\|_{L^\infty(\mathbb{R}^n)} \leq C(T -t)^{-\frac{1}{p-1}}, \quad \forall t \in (0, T) ?
\end{equation}
\noindent By a simple argument based on Duhamel's formula, we can show that the lower bound in \eqref{equ:blrateIn} is always satisfied (see \cite{We81ijm}). For the upper blow-up rate estimate, it is much less simple and requires more work. Practically, we define for all $x_0 \in \mathbb{R}^n$ ($x_0$ may be a blow-up point of $u$ or not) the following \emph{similarity variables} introduced in Giga and Kohn \cite{GKcpam85, GKiumj87, GKcpam89}:
\begin{equation}\label{equ:simivariables}
y = \frac{x-x_0}{\sqrt{T-t}}, \quad s = -\log(T-t), \quad w_{x_0,T} = (T - t)^\frac{1}{p-1}u(x,t).
\end{equation}
Hence $w_{x_0,T}$ satisfies for all $s \geq -\log{T}$ and for all $y \in \mathbb{R}^n$:
\begin{equation}\label{equ:divw1}
\partial_sw_{x_0,T} = \frac{1}{\rho}\text{div}(\rho \nabla w_{x_0,T}) - \frac{w_{x_0,T}}{p-1}  + |w_{x_0,T}|^{p-1}w_{x_0,T} + e^{-\frac{ps}{p-1}}h\left(e^{\frac{s}{p-1}}w_{x_0,T}\right),
\end{equation}
where 
\begin{equation}\label{equ:rho}
\rho(y) = \left(\frac{1}{4\pi}\right)^{n/2}e^{-\frac{|y|^2}{4}}.
\end{equation}
Here, we say that $w: \mathbb{R}^n \times [-\log T, +\infty) \mapsto \mathbb{R}$ is a solution of \eqref{equ:divw1} if $w$ solves \eqref{equ:divw1} and satisfies
\begin{equation}\label{equ:proW}
w, w_s, \nabla w, \nabla ^2 w \; \text{are bounded and continuous on} \; \mathbb{R}^n \times [-\log T, S],\; \forall S < +\infty.
\end{equation}
We can see that the study of $u$ in the neighborhood of $(x_0,T)$ is equivalent to the study of the long-time behavior of $w_{x_0,T}$ and each result for $u$ has an equivalent formulation in term of $w_{x_0,T}$. In particular, the proof of the upper bound in \eqref{equ:blrateIn} is now equivalent to showing that there exists a time $\hat{s} \geq -\log T$ large enough such that
\begin{equation}\label{equ:estWx0}
\|w_{x_0,T}(s)\|_{L^\infty(\mathbb{R}^n)} \leq C, \quad \forall s \geq \hat{s}.
\end{equation}

\noindent We remark that the perturbation term added to equation \eqref{equ:divw1} satisfies the following inequality, 
\begin{equation}\label{equ:estimateH}
j = 0, 1, \quad e^{-\frac{(p-j)s}{p-1}}\left|h^{(j)}\left(e^\frac{s}{p-1}z \right) \right| \leq \frac{C_0}{s^a}\left(|z|^{p-j} + 1\right),\quad \forall s \geq s_0,
\end{equation}
for some $C_0 > 0$ and $s_0 > 0$ (see Lemma \ref{ap:lemmA2} for a proof of this fact).\\

When $h \equiv 0$, Giga and Kohn proved \eqref{equ:estWx0} in \cite{GKiumj87} for $1 < p < \frac{3n + 8}{3n-4}$ or for non-negative initial data (so that the solution is positive everywhere) with sub-critical $p$. Estimate \eqref{equ:estWx0} is extended for all $p$ satisfying \eqref{equ:rangep} without assuming non-negativity for initial data $u_0$ by Giga, Matsui and Sasayama in \cite{GMSiumj04}. The proof written in \cite{GMSiumj04} is strongly based on the existence of the following Lyapunov functional:
\begin{equation}\label{equ:E0}
\mathcal{E}_0[w](s) = \int_{\mathbb{R}^n} \left(\frac{1}{2}|\nabla w|^2  + \frac{1}{2(p-1)}|w|^2 - \frac{1}{p+1}|w|^{p+1}\right)\rho dy.
\end{equation}
Based on this functional, some energy estimates related to this structure and a bootstrap argument given in \cite{Qu99amuc}, the authors in \cite{GMSiumj04} have established the following key integral estimate
\begin{equation}\label{equ:keyIt}
\sup_{s \geq s'}\int_s^{s+1} \|w_{x_0,T}(s)\|_{L^{p+1}(\mathbf{B}_R)}^{(p+1)q}ds \leq C_{q,s'}, \quad \forall q \geq 2,\quad s' > -\log T.
\end{equation}
Since this estimate holds for all $q \geq 2$, we obtain an upper bound for $w_{x_0,T}$ which yields \eqref{equ:estWx0}.\\

When $h \not\equiv 0$, we wonder whether a perturbation of the method of \cite{GMSiumj04} would work for our problem. A key step is to find a Lyapunov functional for equation \eqref{equ:divw1}. Following the method introduced by Hamza and Zaag in \cite{HZnonl12, HZjhde12} for perturbations of the semilinear wave equation, we introduce 
\begin{equation}\label{equ:lyafun}
\mathcal{J}[w](s) = \mathcal{E}[w](s)e^{\frac{\gamma}{a-1}s^{1-a}} + \theta s^{1-a},
\end{equation}
where $\gamma = 8C_0\left(\frac{p+1}{p-1}\right)^2$ and $\theta > 0$ is sufficiently large constant which will be determined later,
\begin{equation}\label{equ:difE}
\mathcal{E}[w] = \mathcal{E}_0[w] + \mathcal{I}[w], \quad \mathcal{I}[w](s) = - e^{-\frac{p + 1}{p-1}s}\int_{\mathbb{R}^n} H\left(e^{\frac{s}{p-1}}w\right)\rho dy,
\end{equation}
with $H(z) = \int_0^z h(\xi)d\xi$.\\
With this introduction, we derive that the functional $\mathcal{J}[w]$ is a decreasing function of time for equation \eqref{equ:divw1}, provided that $s$ is large enough. More precisely, we have the following:
\begin{theo}[\textbf{Existence of a Lyapunov functional for equation \eqref{equ:divw1}}] \label{theo:lya} Let $a, p, n, M$ be fixed, consider $w$ a solution of equation \eqref{equ:divw1} satisfying \eqref{equ:proW}. Then there exist $\hat{s}_0 = \hat{s}_0(a, p, n, M) \geq s_0$ and $\hat{\theta}_0 = \hat{\theta}_0(a,p,n,M)$ such that if $\theta \geq \hat{\theta}_0$, then $\mathcal{J}$ satisfies the following inequality, for all $s_2 > s_1 \geq \max\{\hat{s}_0, -\log T\}$, 
\begin{equation}\label{equ:estimateJinT}
\mathcal{J}[w](s_2) - \mathcal{J}[w](s_1) \leq - \frac{1}{2}\int_{s_1}^{s_2}\int_{\mathbb{R}^n}(\partial_sw)^2\rho dy ds.
\end{equation}
\end{theo}

\noindent As mentioned above, the existence of this Lyapunov functional $\mathcal{J}$ is a crucial step in the derivation of the blow-up rate for equation \eqref{equ:problem}. Indeed, with the functional $\mathcal{J}$ and some more work, we are able to adapt the analysis in \cite{GMSiumj04} for equation \eqref{equ:problem} in the case $h \equiv 0$ and get the following result:
\begin{theo}[\textbf{Blow-up rate for equation \eqref{equ:problem}}] \label{theo:blrate} Let $a, p, n, M$ be fixed, $p$ satisfy \eqref{equ:rangep}. There exists $\hat{s}_1=\hat{s}_1(a, p, n, M) \geq \hat{s}_0$ such that if $u$ is a blow-up solution of equation \eqref{equ:problem} with a blow-up time $T$, then \\
$(i)$ for all $s \geq s' = \max\{\hat{s}_1, -\log T\}$, 
\begin{equation}\label{equ:boundw}
\|w_{x_0,T}(y,s)\|_{L^\infty(\mathbb{R}^n)} \leq C,
\end{equation}
where $w_{x_0,T}$ is defined in \eqref{equ:simivariables} and $C$ is a positive constant depending only on $n, p, M$ and a bound of $\|w_{x_0,T}(\hat{s}_0)\|_{L^\infty}$.\\
$(ii)$ For all $t \in [t_1, T)$ where $t_1 = T - e^{-s'}$,  
\begin{equation}\label{equ:blrateu1}
\|u(x,t)\|_{L^\infty(\mathbb{R}^n)} \leq C (T - t)^{-\frac{1}{p-1}}.
\end{equation}
\end{theo}
\begin{rema} The proof of Theorem \ref{theo:blrate} is far from being a straightforward adaptation of \cite{GMSiumj04}. Indeed, three major difficulties arise in our case and make the heart of our contribution:\\
- the existence of a Lyapunov functional in similarity variables (see Theorem \ref{theo:lya} above),\\
- the control of the $L^2$-norm in terms of the energy (see $(ii)$ of Proposition \ref{prop:boundEpsi}, where we rely on a new blow-up criterion greatly simplifying the approach in \cite{GKiumj87}),\\
- the proof of a nonlinear parabolic result (see Proposition \ref{prop:regpar} below).
\end{rema}

\noindent The estimate obtained in Theorem \ref{theo:blrate} is a fundamental step in studying  the asymptotic behavior of blow-up solutions.  When $h \equiv 0$, Giga and Kohn in \cite{GKiumj87, GKcpam89} (see also \cite{GKcpam85}) obtained the following result: \textit{For a given blow-up point $x_0$, it holds that 
$$\lim_{s \to +\infty} w_{x_0,T}(y,s) = \lim_{t \to T} (T - t)^{\frac{1}{p-1}}u(x_0 + y\sqrt{T-t}, t) = \pm\kappa,$$
where $\kappa = (p-1)^{\frac{1}{p-1}}$, uniformly on compact subsets of $\mathbb{R}^n$. The result is pointwise in $x_0$. Besides, for a.e. $y$, $\lim_{s \to +\infty} \nabla w_{x_0,T}(y,s) = 0$.}\\

\noindent For our problem, when $h \not\equiv 0$ and $h$ is given in \eqref{equ:h}, we also derive an analogous result on the behavior of $w_{x_0,T}$ as $s \to +\infty$. We claim the following:
\begin{theo}[\textbf{Behavior of $w_{x_0,T}$ as $s \to +\infty$}] \label{theo:wtoin} 
Let $a, p, n, M$ be fixed, $p$ satisfy \eqref{equ:rangep}. Consider $u(t)$ a solution of equation \eqref{equ:problem} which blows up at time $T$ and $x_0$ a blow-up point. Then
$$\lim_{t \to T}(T-t)^{\frac{1}{p-1}}u(x_0 + y\sqrt{T-t},t) = \lim_{s \to +\infty}\, w_{x_0,T} (y,s) = \pm\kappa,$$
holds in $L^2_\rho$ ($L^2_\rho$ is the weighted $L^2$ space associated with the weight $\rho$ \eqref{equ:rho}), and also uniformly on each compact subset of $\mathbb{R}^n$.
\end{theo}

\noindent Up to changing $u_0$ in $-u_0$ and $h$ in $-h$, we may assume that $w \to \kappa$ in $L^2_\rho$ as $s \to +\infty$. Let us consider $\phi$ a positive solution  of the associated ordinary differential equation of equation \eqref{equ:divw1}
\begin{equation}\label{equ:phiODE}
\phi_s = -\frac{\phi}{p-1} + \phi^p + e^{-\frac{ps}{p-1}}h\left(e^\frac{s}{p-1}\phi\right)
\end{equation}
such that 
\begin{equation}\label{equ:solphi}
\phi(s) = \kappa + \mathcal{O}\left(\frac{1}{s^a}\right) \quad \text{as} \quad s \to + \infty, 
\end{equation}
(see Lemma \ref{ap:lemmA3} for a proof of the existence of $\phi$).\\
\noindent Let us introduce $v_{x_0,T} = w_{x_0,T} - \phi(s)$, then $\|v_{x_0,T}(y,s)\|_{L^2_\rho} \to 0$ as $s \to +\infty$ and $v_{x_0,T}$ (or $v$ for simplicity) satisfies the following equation:
\begin{equation*}
\partial_s v = (\mathcal{L} + \omega(s))v + F(v) + H(v,s),\quad \forall y \in \mathbb{R}^n, \; \forall s \in [-\log T, +\infty),
\end{equation*}
where $\mathcal{L} = \Delta - \frac{y}{2}\cdot \nabla + 1$ and $\omega$, $F$, $H$ satisfy
$$|\omega(s)| = \mathcal{O}(s^{-a})\quad\text{and} \quad |F(v)| + |H(v,s)| = \mathcal{O}(|v|^2) \quad \text{as $s \to +\infty$},$$
(see the beginning of Section \ref{sec:refasy} for the proper definitions of $\omega$, $F$ and $G$).\\
Since the linear part will play an important role in our analysis, let us point out its properties. The operator $\mathcal{L}$ is self-adjoint on $L^2_\rho(\mathbb{R}^n)$. Its spectrum  is given by
$$spec(\mathcal{L}) = \{1 - \frac{m}{2},\; m \in \mathbb{N}\},$$
and it consists of eigenvalues. The eigenfunctions of $\mathcal{L}$ are derived from Hermite polynomials:\\
- For $n = 1$, the eigenfunction corresponding to $1 - \frac{m}{2}$ is
\begin{equation}\label{equ:eigenfu1}
h_m(y) = \sum_{k= 0}^{\left[\frac{m}{2}\right]} \frac{m!}{k!(m - 2k)!}(-1)^ky^{m - 2k}, 
\end{equation}
- For $n \geq 2$: we write the spectrum of $\mathcal{L}$ as 
$$spec(\mathcal{L}) = \{ 1 - \frac{|m|}{2},\; |m| = m_1 + \dots + m_n, \;(m_1,\dots, m_n) \in \mathbb{N}^n\}.$$
For $m = (m_1, \dots, m_n) \in \mathbb{N}^n$, the eigenfunction corresponding to $1 - \frac{|m|}{2}$ is 
\begin{equation}\label{equ:eigenfu}
H_m(y) = h_{m_1}(y_1)\dots h_{m_n}(y_n),
\end{equation}
where $h_{m}$ is defined in \eqref{equ:eigenfu1}.\\

\noindent By studying the behavior of $v$ as $s \to + \infty$, we obtain the following result:
\begin{theo}[\textbf{Classification of the behavior of $w$ as $s \to +\infty$}]\label{theo:refinedasymptotic} Consider $u(t)$ a solution of equation \eqref{equ:problem} which blows-up at time $T$ and $x_0$ a blow-up point. Let $w(y,s)$ be a solution of equation \eqref{equ:divw1}. Then one of the following possibilities occurs:\\
$i)\;$ $w(y,s) \equiv \phi(s)$,\\
$ii)$ There exists $l \in \{1, \dots, n\}$ such that up to an orthogonal  transformation of coordinates, we have
$$
w(y,s) = \kappa -\frac{\kappa}{4ps} \left(\sum_{j=1}^l y_j^2 -2 l\right) + \mathcal{O}\left(\frac{1}{s^a}\right) + \mathcal{O}\left(\frac{\log s}{s^2} \right) \quad \text{as} \quad s \to +\infty.
$$
$iii)$ There exist an  integer number $m \geq 3$ and constants $c_\alpha$ not all zero such that 
$$w(y,s) = \phi(s) - e^{-\left(\frac{m}{2}-1\right)s}\sum_{|\alpha| = m}c_\alpha H_\alpha(y) + o\left(e^{-\left(\frac{m}{2}-1\right)s}\right) \quad \text{as}\quad s \to +\infty. $$
The convergence takes place in $L^2_\rho$ as well as in $\mathcal{C}^{k,\gamma}_{loc}$ for any $k \geq 1$ and some $\gamma \in (0,1)$.
\end{theo}
\begin{rema}
Applying our result to a space-independent solution of \eqref{equ:divw1}, we get the uniqueness of the solution of the ODE \eqref{equ:phiODE} that converges to $\kappa$ as $s \to +\infty$.
\end{rema}
\begin{rema}
Since both the perturbed ($h \not \equiv 0$) and the unperturbed ($h \equiv 0$) cases in equation \eqref{equ:problem} share the same convergence stated in Theorem \ref{theo:refinedasymptotic}, we wonder whether the perturbation $h$ may have an influence on further terms of the expansion of $w$. From our result, if case $(ii)$ occurs, we see no difference in the following term of the expansion. On the contrary, if case $(i)$ or $(iii)$ occurs, with $h(x) = \mu \frac{|x|^{p-1}x}{\log(2 + x^2)}$, we see from Lemma \ref{ap:lemmA3} that $$w(y,s) - \kappa \sim \frac{C_0(a,p,\mu)}{s^a} \quad \text{as $s \to +\infty$},$$
which is clearly different from the unperturbed case when in case $(i)$, we have $w \equiv \kappa$ and case $(iii)$, we have $w - \kappa = \mathcal{O}(e^{-s})$, (see \cite{HVaihn93}, \cite{VELtams93}).
\end{rema}
\begin{rema} If we linearize $w$ around $\kappa$, which is an explicit profile, we then fall in logarithmic scales $\mu = \frac{1}{|\log \epsilon|}$ with $\epsilon = T-t$. Further refinements in this direction should give an expansion of $w - \kappa$ in terms of powers of $\mu$, i.e in logarithmic scales of $\epsilon$. Therefore, we can not reach significantly small error terms in the expansion of the solution $w$ as $(iii)$ of Theorem \ref{theo:refinedasymptotic} describes. In order to escape this situation, a relevant approximation is required in order to go beyond all logarithmic scales, i.e approximations up to lower order terms such as $\epsilon^\alpha$ for some $\alpha > 0$. Our idea to capture such relevant terms is to abandon the explicit profile obtained as a first order approximation, namely $\kappa$, and take an implicit profile function as a first order description of the singular behavior, namely $\phi(s)$ introduced in \eqref{equ:phiODE} and \eqref{equ:solphi}. A similar idea was used by Zaag \cite{ZAAcmp02} where the solution was linearized around a less explicit profile function in order to go beyond all logarithmic scales. For our problem, we particularly take $\phi(s)$ as the implicit profile function, which is a solution of the associated ODE of equation \eqref{equ:divw1} in $w$ such that $\phi(s) \to \kappa$ as $s \to +\infty$. By linearizing the solution $w$ around $\phi$, we can get to error terms of polynomial order $\epsilon^{\left(\frac{m}{2} - 1\right)}$, as stated in $(iii)$ of Theorem \ref{theo:refinedasymptotic}. 
\end{rema}

\begin{rema} When $h(x) = |x|^q$ with $q \in (1,p)$, we see that 
\begin{equation*}
\phi(s) -\kappa \sim C_0'(p,q)e^{-\lambda s} \quad \text{as} \quad s \to +\infty.
\end{equation*}
If case $(ii)$ in Theorem \ref{theo:refinedasymptotic} holds, we then recover the same expansion as in the unperturbed case $(h \equiv 0)$. On the contrary, if case $(i)$ or $(iii)$ occurs, then
\begin{equation*}
w(y,s) -\kappa \sim C_0'(p,q)e^{-\lambda s} \quad \text{as} \quad s \to +\infty.
\end{equation*}
Moreover, if case $(iii)$ in Theorem \ref{theo:refinedasymptotic} holds, we have  new terms in the expansion of $w$ which was not available in the unperturbed case, namely
$$w(y,s) = \kappa - \sum_{k =1}^{K}C_ke^{-k\lambda s} - e^{-\left(\frac{m}{2}-1\right)s} \sum_{|\alpha| = m}c_\alpha H_\alpha(y) + o\left(e^{-\left( \frac{m}{2}-1\right)s}\right)\;\text{as $s \to +\infty$},$$
where $C_k, k = 1, 2,\dots, K$ are some constants depending on $p$ and $q$, and $K\in \mathbb{N}$ is the integer part of $\frac{1}{\lambda}\left(\frac{m}{2} - 1\right)$.
\end{rema}

In the last section, we will extend the asymptotic behavior of $w$ obtained in Theorem \ref{theo:refinedasymptotic} to larger regions. Particularly, we claim the following:
\begin{theo}[\textbf{Convergence extension of $w_a$ to larger regions}] \label{theo:pro} For all $K_0 > 0$,\\
$i)$ if $ii)$ of Theorem \ref{theo:refinedasymptotic} occurs, then
\begin{equation}
\sup_{|\xi| \leq K_0} \left|w(\xi \sqrt{s},s) - f_l(\xi)\right| = \mathcal{O}\left(\frac{1}{s^{a-1}}\right) + \mathcal{O}\left(\frac{\log s}{s}\right)\quad \text{as} \; s \to +\infty, 
\end{equation}
where  
$$f_l(\xi) = \kappa\left(1 + \frac{p-1}{4p} \sum_{j=1}^l \xi_j^2 \right)^{-\frac{1}{p-1}}, \quad \forall \xi \in \mathbb{R}^n,$$
with $l$ the same as in $ii)$ of Theorem \ref{theo:refinedasymptotic}.\\
\noindent $ii)$ if $iii)$ of Theorem \ref{theo:refinedasymptotic} occurs, then $m \geq 4$ is even, and 
\begin{equation}
\sup_{|\xi| \leq K_0} \left|w\left(\xi e^{\left(\frac{1}{2} - \frac{1}{m}\right)s}\right) - \psi_m(\xi)\right| \to 0  \quad \text{as} \; s \to +\infty,
\end{equation}
where 
$$\psi_m(\xi) = \kappa\left(1 + \kappa^{-p} \sum_{|\alpha| = m}c_\alpha\xi^\alpha\right)^{-\frac{1}{p-1}}, \quad \forall \xi \in \mathbb{R}^n,$$
with $c_\alpha$ the same as in Theorem \ref{theo:refinedasymptotic}.
\end{theo}
 
\noindent Let us mention briefly the structure of the paper. In Section \ref{sec:Lya}, we prove the existence of Lyapunov functional for equation \eqref{equ:divw1} (Theorem \ref{theo:lya}), we then get Theorem \ref{theo:blrate} and Theorem \ref{theo:wtoin}. In Section \ref{sec:refasy}, we follow the method of \cite{FLaihn93} and \cite{VELtams93} to prove Theorem \ref{theo:refinedasymptotic}. Finally, the section \ref{sec:ext} is devoted to the proof of Theorem \ref{theo:pro}.\\

\noindent \textbf{Acknowledgement:} The author is grateful to H. Zaag for his dedicated advice, suggestions and remarks during the preparation of this paper.

\section{A Lyapunov functional}\label{sec:Lya}
This section is divided in four subsections: we first prove the existence of a Lyapunov functional for equation \eqref{equ:divw1} (Theorem \ref{theo:lya}); after that, we derive a blow-up criterion for equation \eqref{equ:divw1} and some energy estimates based on this Lyapunov functional. Following the method of \cite{GMSiumj04}, we prove the boundedness of solution in similarity variables which determines the blow-up rate for solution of \eqref{equ:problem} (Theorem \ref{theo:blrate}). Finally, we derive the limit of $w$ as $s \to + \infty$, which concludes Theorem \ref{theo:wtoin}.\\

\noindent In what follows, we denote by $C_i, i = 0, 1, \dots$ positive constants depending only on $a, n, p, M$, and by $L^{q}_{\rho}(\Omega)$ the weighted $L^q(\Omega)$ space endowed with the norm
$$\|\varphi\|_{L^{q}_{\rho}(\Omega)} = \left(\int_{\Omega}|\varphi(y)|^q\rho(y) dy\right)^\frac{1}{q},$$
and by $H^1_\rho(\Omega)$ the space of function $\varphi \in L^2_\rho(\Omega)$ satisfying $\nabla \varphi \in L^2_\rho(\Omega)$, endowed with the norm
$$\|\varphi\|_{H^1_\rho(\Omega)} = \left(\|\varphi\|_{L^{2}_{\rho}(\Omega)}^2 + \frac{1}{p-1}\left\| \nabla \varphi\right\|_{L^{2}_{\rho}(\Omega)}^2\right)^\frac{1}{2}.$$
We denote by $\mathbf{B}_R(x)$ the open ball in $\mathbb{R}^n$ with center $x$ and radius $R$, and set $\mathbf{B}_R := \mathbf{B}_R(0)$.

\subsection{Existence of a Lyapunov function}
In this part, we aim at proving that the functional $\mathcal{J}$ defined in \eqref{equ:lyafun} is a Lyapunov functional for equation \eqref{equ:divw1}. Note that that functional is far from being trivial and it is our main contribution. We first claim the following lemma:
\begin{lemm} \label{lemm:lya} Let $a, p, n, M$ be fixed and $w$ be solution of equation \eqref{equ:divw1} satisfying \eqref{equ:proW}. There exists $\tilde{s}_0 = \tilde{s}_0(a, p, n, M) \geq s_0$ such that the functional of $\mathcal{E}$ defined in \eqref{equ:difE} satisfies  the following inequality, for all $s \geq \max\{\tilde{s}_0, -\log T\}$, 
\begin{equation}\label{equ:estimateDE}
\frac{d}{ds}\mathcal{E}[w](s) \leq - \frac{1}{2}\int_{\mathbb{R}^n}w_s^2\rho dy + \gamma s^{-a}\mathcal{E}[w](s) + Cs^{-a},
\end{equation}
where $\gamma = 8C_0\left(\frac{p+1}{p-1}\right)^2$, $C_0$ is introduced in \eqref{equ:estimateH} and $C$ is a positive constant depending only on $a, p, n , M$. 
\end{lemm}
\noindent Let us first derive Theorem \ref{theo:lya} from Lemma \ref{lemm:lya} which will be proved later.
\begin{proof}[\textbf{Proof of Theorem \ref{theo:lya} admitting Lemma \ref{lemm:lya}}]
Differentiating the functional $\mathcal{J}$ defined in \eqref{equ:lyafun}, we obtain 
\begin{align*}
\frac{d}{ds}\mathcal{J}[w](s) &= \frac{d}{ds}\left\{\mathcal{E}[w](s)e^{\frac{\gamma}{a-1}s^{1-a}} + \theta s^{1-a}\right\}\\
& = \frac{d}{ds}\mathcal{E}[w](s)e^{\frac{\gamma}{a-1}s^{1-a}} - \gamma s^{-a}\mathcal{E}[w](s)e^{\frac{\gamma}{a-1}s^{1-a}} - (a-1)\theta s^{-a}\\
& \leq - \frac{1}{2} e^{\frac{\gamma}{a-1}s^{1-a}}\int_{\mathbb{R}^n}w_s^2\rho dy  + \left[C e^{\frac{\gamma}{a-1}s^{1-a}} - (a-1)\theta\right]s^{-a} \quad \text{(use \eqref{equ:estimateDE}).}
\end{align*}
Choosing $\theta$ large enough such that $C e^{\frac{\gamma}{a-1}{\tilde{s}_0}^{1-a}} - (a-1)\theta \leq 0$ and noticing that $e^{\frac{\gamma}{a-1}s^{1-a}} \geq 1$ for all $s > 0$, we derive 
$$
\frac{d}{ds}\mathcal{J}[w](s) \leq -\frac{1}{2} \int_{\mathbb{R}^n}w_s^2\rho dy, \quad \forall s \geq \tilde{s}_0.
$$
This implies inequality \eqref{equ:estimateJinT} and concludes the proof of Theorem \ref{theo:lya}, assuming that Lemma \ref{lemm:lya} holds.
\end{proof}
\noindent It remains to prove Lemma \ref{lemm:lya} in order to conclude the proof of Theorem \ref{theo:lya}.
\begin{proof}[\textbf{Proof of Lemma \ref{lemm:lya} }] 
Multiplying equation \eqref{equ:divw1} with $w_s\rho$ and integrating by parts:
\begin{align*}
\int_{\mathbb{R}^n} |w_s|^2\rho = -\frac{d}{ds} \left\{\int_{\mathbb{R}^n} \left( \frac{1}{2}|\nabla w|^2 + \frac{1}{2(p-1)} |w|^2 - \frac{1}{p+1}|w|^{p+1}\right)\rho dy \right\}&\\
+ e^{-\frac{ps}{p-1}}\int_{\mathbb{R}^n} h\left(e^{\frac{s}{p-1}}w\right)w_s\rho dy&.
\end{align*}

\noindent For the last term of the above expression, denoting $H(z) = \int_0^z h(\xi)d\xi$, we write in the following:
\begin{align*}
e^{-\frac{ps}{p-1}}\int_{\mathbb{R}^n} h\left(e^{\frac{s}{p-1}}w\right)w_s\rho dy = e^{-\frac{(p+1)s}{p-1}} \int_{\mathbb{R}^n} h\left(e^{\frac{s}{p-1}}w\right)\left(e^{\frac{s}{p-1}} w_s + \frac{e^{\frac{s}{p-1}}}{p-1}w \right)\rho dy&\\
- \frac{1}{p-1}e^{-\frac{ps}{p-1}}\int_{\mathbb{R}^n} h\left(e^{\frac{s}{p-1}}w\right)w\rho dy&\\
= e^{-\frac{p+1}{p-1}s} \frac{d}{ds} \int_{\mathbb{R}^n} H\left(e^{\frac{s}{p-1}}w\right)\rho dy- \frac{1}{p-1} e^{-\frac{ps}{p-1}}\int_{\mathbb{R}^n} h\left(e^{\frac{s}{p-1}}w\right)w\rho dy&.
\end{align*}
This yields 
\begin{align*}
\int_{\mathbb{R}^n} |w_s|^2\rho dy = -\frac{d}{ds} \left\{\int_{\mathbb{R}^n} \left( \frac{1}{2}|\nabla w|^2 + \frac{1}{2(p-1)} |w|^2 - \frac{1}{p+1} |w|^{p+1}\right)\rho dy \right\}&\\
+ \frac{d}{ds} \left\{e^{-\frac{p + 1}{p-1}s}\int_{\mathbb{R}^n} H\left(e^{\frac{s}{p-1}}w\right)\rho dy\right\}&\\
+ \frac{p+1}{p-1}e^{-\frac{p + 1}{p-1}s}\int_{\mathbb{R}^n} H\left(e^{\frac{s}{p-1}}w\right)\rho dy&\\
- \frac{1}{p-1} e^{-\frac{ps}{p-1}}\int_{\mathbb{R}^n} h\left(e^{\frac{s}{p-1}}w\right)w\rho dy&. 
\end{align*}
From the definition of the functional $\mathcal{E}$ given in \eqref{equ:difE}, we derive a first identity in the following: 
\begin{align}
\frac{d}{ds}\mathcal{E}[w](s) = -\int_{\mathbb{R}^n} |w_s|^2\rho dy + \frac{p+1}{p-1}e^{-\frac{p + 1}{p-1}s}\int_{\mathbb{R}^n} H\left(e^{\frac{s}{p-1}}w\right)\rho dy &\nonumber\\
 - \frac{1}{p-1} e^{-\frac{ps}{p-1}}\int_{\mathbb{R}^n} h\left(e^{\frac{s}{p-1}}w\right)w\rho dy&. \label{equ:Id1}
\end{align}
A second identity is obtained by  multiplying equation \eqref{equ:divw1} with $w\rho$ and integrating by parts:
\begin{align*}
\frac{d}{ds}\int_{\mathbb{R}^n} |w|^2\rho dy  = -4\left\{\int_{\mathbb{R}^n} \left( \frac{1}{2}|\nabla w|^2 + \frac{1}{2(p-1)} |w|^2 - \frac{1}{p+1} |w|^{p+1}\right) \rho dy \right.&\\
\left.- e^{-\frac{(p + 1)s}{p-1}}\int_{\mathbb{R}^n} H\left(e^{\frac{s}{p-1}}w\right)\rho dy\right\}&\\
+ \left(2 - \frac{4}{p+1}\right)\int_{\mathbb{R}^n} |w|^{p+1}\rho dy - 4e^{-\frac{p + 1}{p-1}s}\int_{\mathbb{R}^n} H\left(e^{\frac{s}{p-1}}w\right)\rho dy &\\
+ 2e^{-\frac{ps}{p-1}}\int_{\mathbb{R}^n} h\left(e^{\frac{s}{p-1}}w\right)w\rho dy&.
\end{align*}
Using again the definition of $\mathcal{E}$ given in \eqref{equ:difE}, we derive the second identity in the following:
\begin{align}
\frac{d}{ds} \int_{\mathbb{R}^n} |w|^2\rho dy & = -4 \mathcal{E}[w](s) + 2\frac{p-1}{p+1}\int_{\mathbb{R}^n} |w|^{p+1}\rho dy\nonumber\\
&-  4e^{-\frac{p + 1}{p-1}s}\int_{\mathbb{R}^n} H\left(e^{\frac{s}{p-1}}w\right)\rho dy + 2e^{-\frac{ps}{p-1}}\int_{\mathbb{R}^n} h\left(e^{\frac{s}{p-1}}w\right)w\rho dy.\label{equ:Id2}
\end{align}
From \eqref{equ:Id1}, we estimate
\begin{align*}
\frac{d}{ds}\mathcal{E}[w](s) &\leq -\int_{\mathbb{R}^n} |w_s|^2\rho dy\\
& + \frac{p+1}{p-1}\int_{\mathbb{R}^n}\left\{ \left|e^{-\frac{p + 1}{p-1}s} H\left(e^{\frac{s}{p-1}}w\right)\right| + \left|e^{-\frac{ps}{p-1}} h\left(e^{\frac{s}{p-1}}w\right)w \right| \right\}\rho dy. 
\end{align*}
From \eqref{equ:estimateH} and using the fact that $|w| \leq |w|^{p+1} + 1$, we obtain for all $s \geq s_0$,
\begin{equation}\label{equ:estimateHhe}
\left|e^{-\frac{p + 1}{p-1}s}  H\left(e^{\frac{s}{p-1}}w\right)\right| + \left|e^{-\frac{p}{p-1}s} h\left(e^{\frac{s}{p-1}}w\right)w \right| \leq 2C_0 s^{-a}\left(|w|^{p+1} + 1\right).
\end{equation}
Using \eqref{equ:estimateHhe} yields 
\begin{equation}\label{equ:estDE3}
\frac{d}{ds}\mathcal{E}[w](s) \leq -\int_{\mathbb{R}^n} |w_s|^2\rho dy + C_1s^{-a}\int_{\mathbb{R}^n} |w|^{p+1}\rho dy + C_1s^{-a},
\end{equation}
where $C_1 = 2C_0\frac{p+1}{p-1}$.\\
From \eqref{equ:Id2}, we have
\begin{align*}
\int_{\mathbb{R}^n} |w|^{p+1}\rho dy &\leq \frac{2(p+1)}{p-1}\mathcal{E}[w](s) + \frac{p+1}{p-1}\int_{\mathbb{R}^n} |w_s w| \rho dy \\
& \quad +   \frac{2(p+1)}{p-1}\int_{\mathbb{R}^n}  \left( \left| e^{-\frac{p + 1}{p-1}s} H\left(e^{\frac{s}{p-1}}w\right) \right|+ \left | e^{-\frac{ps}{p-1}} h\left(e^{\frac{s}{p-1}}  w\right)w \right| \rho dy \right).
\end{align*}
Using the fact that $|w_sw| \leq \epsilon (|w_s|^2 + |w|^{p+1}) + C_2(\epsilon)$ for all $\epsilon > 0$ and \eqref{equ:estimateHhe}, we obtain 
\begin{align*}
\int_{\mathbb{R}^n} |w|^{p+1}\rho dy &\leq  \frac{2(p+1)}{p-1}\mathcal{E}[w](s) + \epsilon' \int_{\mathbb{R}^n} |w_s|^2\rho dy  \\
& \quad + \left(\epsilon' + 2C_1s^{-a}\right)\int_{\mathbb{R}^n} |w|^{p+1}\rho dy + 2C_1s^{-a} + C_3,
\end{align*}
where $\epsilon' = \epsilon\frac{p+1}{p-1}$, $C_3 = 2C_1 + C_2\frac{p+1}{p-1}$.\\
Taking $\epsilon = \frac{p-1}{4(p+1)}$ and $s_1$ large enough such that $2C_1s^{-a} \leq \frac{1}{4}$ for all $s \geq s_1$, we see that
\begin{equation}\label{equ:estWp1}
\int_{\mathbb{R}^n} |w|^{p+1}\rho dy \leq \frac{4(p+1)}{p-1}\mathcal{E}[w](s) + \frac{1}{2}\int_{\mathbb{R}^n} |w_s|^2\rho dy  + C_4, \quad \forall s > s_1,
\end{equation}
with $C_4 = \frac{C_3}{2} + \frac{1}{8}$.\\
Substituting \eqref{equ:estWp1} into \eqref{equ:estDE3} yields \eqref{equ:estimateDE} with  $\tilde{s}_0 = \max\{s_0,s_1\}$. This concludes the proof of Lemma \ref{lemm:lya}. Since we have already showed that Theorem \ref{theo:lya} is a direct consequence of Lemma \ref{lemm:lya}, this is also the conclusion of Theorem \ref{theo:lya}.
\end{proof}

\subsection{A blow-up criterion for the equation in similarity variables}
In this part, we give a new blow-up criterion for equation \eqref{equ:divw1}. Then, we will use it to control the $L^2$-norm in terms of the energy (see $(ii)$ of Proposition \ref{prop:boundEpsi}). We claim the following:
\begin{lemm}\label{lemm:cri}
Let $a, p, n, M$ be fixed and $w$ be solution of equation \eqref{equ:divw1} satisfying \eqref{equ:proW}. If there exists $\tilde{s}_1 =\tilde{s}_1(a, p, n, M) \geq \max\{\hat{s}_0, -\log T\}$ such that
\begin{equation}\label{equ:bloCri}
-4\mathcal{J}[w](\bar{s})  + \frac{p-1}{p+1}\left(\int_{\mathbb{R}^n} |w(y,\bar{s})|^2\rho dy\right)^\frac{p+1}{2} > 0 \quad \text{for some} \; \bar{s} \geq \tilde{s}_1,
\end{equation}
then $w$ is not defined for all $(y,s) \in \mathbb{R}^n\times[\bar{s}, +\infty)$.
\end{lemm}

\begin{proof} We proceed by contradiction and suppose that $w$ is defined for all $s \in [\bar{s}, +\infty)$. From definition of $\mathcal{J}$ in \eqref{equ:lyafun} and from  \eqref{equ:Id2}, \eqref{equ:estimateHhe}, we have  for all $s \geq s_0$,
\begin{align}
\frac{d}{ds} \int_{\mathbb{R}^n} |w|^2\rho dy  &\geq -4 e^{-\frac{\gamma}{a-1}s^{1-a}}\left(\mathcal{J}[w](s) - \theta s^{1-a}\right)\nonumber\\
& + 2\left(\frac{p-1}{p+1} - 4C_0s^{-a}\right)\int_{\mathbb{R}^n} |w|^{p+1}\rho dy  - 8C_0s^{-a}.\label{equ:ww1}
\end{align}
We take $s_1$ large enough such that 
$$4C_0s^{-a} \leq \frac{p-1}{2(p+1)} \quad \text{and} \quad e^{-\frac{\gamma}{a-1}s^{1-a}} - \frac{2C_0}{s} > 0 \quad \text{for all} \quad s \geq s_1.$$
Then, using Jensen's inequality and noting that $e^{-\frac{\gamma}{a-1}s^{1-a}} \leq 1$ for all $s > 0$, we get from \eqref{equ:ww1} the following: for all $s \geq \max\{0, s_0, s_1\}$,
\begin{equation}\label{equ:S1}
\frac{d}{ds} \int_{\mathbb{R}^n} |w|^2\rho dy  \geq -4\mathcal{J}[w](s) + \frac{p-1}{p+1}\left(\int_{\mathbb{R}^n} |w|^2\rho dy\right)^\frac{p+1}{2}.
\end{equation}
Setting $f(s) = \int_{\mathbb{R}^n} |w(y,s)|^2\rho dy$, $A = -4\mathcal{J}[w](\bar{s})$ and $B = \frac{p-1}{p+1}$, then using the fact that $\mathcal{J}$ is decreasing in time to get that
$$f'(s) \geq A + Bf(s)^\frac{p+1}{2}, \quad \forall s \geq \bar{s}.$$
The hypothesis reads $A + Bf(\bar{s})^\frac{p+1}{2} > 0$ which implies that 
$$f'(s) > 0 \quad \text{and} \quad A + Bf(s)^\frac{p+1}{2} > 0, \quad \forall s \geq \bar{s}.$$
By a direct integration, we obtain
$$\forall s \geq \bar{s}, \quad s - \bar{s} \leq \int_{f(\bar{s})}^{f(s)}\frac{dz}{A + Bz^\frac{p+1}{2}} \leq  \int_{f(\bar{s})}^{+\infty}\frac{dz}{A + Bz^\frac{p+1}{2}} < +\infty,$$
which is a contradiction and Lemma \ref{lemm:cri} is proved.

\end{proof}

As a consequence of Theorem \ref{theo:lya} and Lemma \ref{lemm:cri}, we obtain the following estimates which will be useful for getting Theorem \ref{theo:blrate}:
\begin{prop} \label{prop:boundEpsi} Let $w$ be solution of equation \eqref{equ:divw1} satisfying \eqref{equ:proW}, it holds that
\begin{equation*}
- Q_0 \leq \mathcal{E}[w](s) \leq 2J_0, \quad \forall s \geq \tilde{s}_2 = \max\{\hat{s}_0, -\log T\},
\end{equation*}
where $J_0 = \mathcal{J}[w](\tilde{s}_2)$ and $Q_0 = \theta \tilde{s}_2^{1-a}$. Moreover, there exists a time $\tilde{s}_3 \geq \max\{\hat{s}_0, -\log T\}$ such that for all $s \geq \tilde{s}_3$
\begin{align*}
& (i)  \quad \int_{s}^{s + 1} \left\|w_\tau(\tau)\right\|_{L^2_{\rho}(\mathbb{R}^n)}^2 d\tau\leq 2J_0, \\
& (ii) \quad \|w(s)\|_{L_\rho^2(\mathbb{R}^n)}^2 \leq J_1,\\
& (iii)\quad \|w(s)\|^{p+1}_{L^{p+1}_\rho(\mathbb{R}^n)} \leq J_2\left(1 + \|w(s)\|^{2}_{H^1_\rho(\mathbb{R}^n)}\right), \\
& (iv) \quad \|w(s)\|^{2}_{H^1_\rho(\mathbb{R}^n)} \leq J_3\left(1 + \|w_s(s)\|_{L_\rho^2(\mathbb{R}^n)}\right),\\
&(v)\quad \int_s^{s+1}\left\|w(\tau) \right\|_{L^{p+1}_{\rho}(\mathbb{R}^n)}^{2(p+1)} d\tau \leq J_4,\\
&(vi)\quad \int_s^{s+1}\left\|w(\tau) \right\|_{H^1_{\rho}(\mathbb{R}^n)}^{2} d\tau \leq J_5,
\end{align*}
where $J_i, \, i = 1,\dots, 5$ depend only on $J_0, Q_0, a, p, n, M$.
\end{prop}
\begin{proof} The upper and lower bounds of $\mathcal{E}$, $(i)$ and $(ii)$ obviously follow from Theorem \ref{theo:lya} and Lemma \ref{lemm:cri} (in fact, since $w$ is defined for all $s \geq \tilde{s}_1$, condition \eqref{equ:bloCri} is never satisfied).

\noindent$(iii)$ By definition of $\mathcal{E}$ given in \eqref{equ:difE} and \eqref{equ:estimateHhe}, we get for all $s \geq \max\{s_0, -\log T\}$,
\begin{align*}
2\mathcal{E}[w](s) &\leq  \int_{\mathbb{R}^n} \left(|\nabla w|^2 + \frac{1}{p-1}|w|^2 \right)\rho dy\\
& - 2 \left(\frac{1}{p-1} - C_0s^{-a}\right)\int_{\mathbb{R}^n} |w|^{p+1}\rho dy + 2C_0s^{-a}.
\end{align*}
Let $s_1$ large enough such that for all $s \geq s_1$, $C_0s^{-a} \leq \frac{1}{2(p-1)}$, then for all $s \geq \max\{s_0,s_1,-\log T\}$,
\begin{equation*}
2\mathcal{E}[w](s) \leq  \int_{\mathbb{R}^n} \left(|\nabla w|^2 + \frac{1}{p-1}|w|^2 \right)\rho dy - \frac{1}{p-1}\int_{\mathbb{R}^n} |w|^{p+1}\rho dy + \frac{2}{p-1}.
\end{equation*}
This follows that for all $s \geq \max\{s_0,s_1,-\log T\}$,
$$\left\| w(s) \right\|_{L_\rho^{p+1}(\mathbb{R}^n)}^{p+1} \leq -2(p-1)\mathcal{E}[w](s) + (p-1)\left\| w (s) \right\|_{H_\rho^1(\mathbb{R}^n)}^{2} + 1.$$
Since $\mathcal{E}$ is bounded from below, then $(iii)$ follows. \\
$(iv)$ From the definition $\mathcal{E}$ in \eqref{equ:difE}, \eqref{equ:Id2} and \eqref{equ:estimateHhe}, we have $\forall s \geq \max\{s_0, -\log T\}$,
\begin{align*}
\left\|w(s) \right\|_{H_\rho^1(\mathbb{R}^n)}^2 &\leq \frac{1}{p-1}\frac{d}{ds} \int_{\mathbb{R}^n} |w|^2\rho dy + \frac{2(p+1)}{p-1} \mathcal{E}[w](s)\\
&\quad + \frac{4C_0(p+1)}{p-1}s^{-a}\left\| w(s) \right\|_{L_\rho^{p+1}(\mathbb{R}^n)}^{p+1} + \frac{4C_0(p+1)}{p-1}s^{-a}.
\end{align*}
Using $(iii)$, we have for all $s \geq \tilde{s}_3$,
\begin{align*}
\left\|w(s) \right\|_{H_\rho^1(\mathbb{R}^n)}^2 &\leq \frac{1}{p-1}\frac{d}{ds} \int_{\mathbb{R}^n} |w|^2\rho dy + \frac{2(p+1)}{p-1} \mathcal{E}[w](s)\\
&\quad + \frac{4C_0J_2(p+1)}{p-1}s^{-a}\left(1 + \left\| w(s)\right\|_{H_\rho^1(\mathbb{R}^n)}^{2}\right) + \frac{4C_0(p+1)}{p-1}s^{-a}.
\end{align*}
Let $s_2$ large enough such that $\frac{4C_0J_2(p+1)}{p-1}s^{-a} \leq \frac{1}{2}$ for all $s \geq s_2$ and noting that $\mathcal{E}(s)$ is bounded from above, we obtain for all $s \geq \max\{s_2, \tilde{s}_2\}$,
$$\left\|w(s) \right\|_{H_\rho^1(\mathbb{R}^n)}^2 \leq \frac{4}{p-1} \int_{\mathbb{R}^n} |ww_s|\rho dy + C_1,$$
where $C_1 = \frac{4J_0(p+1)}{p-1} + \frac{1}{J_2}$.\\
Using Schwarz's inequality and $(ii)$ yields
$$\left\|w(s) \right\|_{H_\rho^1(\mathbb{R}^n)}^2 \leq \frac{4}{p-1}\|w(s)\|_{L^2_\rho(\mathbb{R}^n)}\|w_s(s)\|_{L^2_\rho(\mathbb{R}^n)} + C_1 \leq \frac{4\sqrt{J_1}}{p-1}\|w_s(s)\|_{L^2_\rho(\mathbb{R}^n)} + C_1,$$
which follows $(iv)$.\\
Since $(v)$ and $(vi)$ follows directly from  $(i)$ and $(iii), (iv)$, we end the proof of Proposition \ref{prop:boundEpsi}.
\end{proof}

\subsection{Boundedness of the solution in similarity variables}
This section is devoted to the proof of Theorem \ref{theo:blrate}, which is a direct consequence of the following theorem:
\begin{theo}\label{theo:blrate1}  Let $a, p, n, M$ be fixed, $p$ satisfy \eqref{equ:rangep}. There exists $\hat{s}_1=\hat{s}_1(a, p, n, M) \geq \hat{s}_0$ such that if $u$ is a blow-up solution of equation \eqref{equ:problem} with a blow-up time $T$, then for all $s \geq s'=\max\{\hat{s}_1, -\log T\}$, 
\begin{equation}\label{equ:boundw1}
\|w_{x_0,T}(y,s)\|_{L^\infty(\mathbf{B}_R)} \leq C,
\end{equation}
where $C$ is a positive constant depending only on $n, p, M, R$ and a bound of $\|w_{x_0,T}(\hat{s}_0)\|_{L^\infty}$.
\end{theo}
\noindent Let us show that Theorem \ref{theo:blrate} follows from Theorem \ref{theo:blrate1}.
\begin{proof}[\textbf{Proof of Theorem \ref{theo:blrate} admitting Theorem \ref{theo:blrate1}}] We have from \eqref{equ:boundw1} that 
$$|w_{x_0,T}(0,s)| \leq C, \quad \forall s \geq s',$$
with $C$ independent on $x_0 \in \mathbb{R}^n$. Therefore, we get from \eqref{equ:simivariables} that 
$$|u(x_0,t)| \leq C(T - t)^{-\frac{1}{p-1}}, \quad \forall x_0 \in \mathbb{R}^n, \forall t \in[T - e^{-s'}, T),$$
which is the conclusion of Theorem \ref{theo:blrate}, assuming that Theorem \ref{theo:blrate1} holds.
\end{proof}
\noindent Following the method in \cite{GMSiumj04}, the proof of Theorem \ref{theo:blrate1} requests the following key integral estimate: 
\begin{lemm}[\textbf{Key integral estimate}] \label{lemm:keyint}
Let $a, p, n, M$ be fixed and $w$ be solution of equation \eqref{equ:divw1} satisfying \eqref{equ:proW}. For all $q \geq 2$ and $R > 0$, there exists $\hat{s}_2 \geq \tilde{s}_3$ and a positive constant $K_q$ such that, 
\begin{equation}\label{equ:keyint}
\int_s^{s+1} \|w(\tau)\|_{L^{p+1}(\mathbf{B}_R)}^{q(p+1)} d\tau \leq K_q ,\quad \forall s \geq \hat{s}_2,
\end{equation}
where $K_q$ depends only on $J_0, Q_0, a, n, p, q, R, \hat{s}_2$.
\end{lemm}
\noindent Let us first show that how Theorem \ref{theo:blrate1} follows from Lemma \ref{lemm:keyint}, then we will prove it later. In order to derive uniform bound in Theorem \ref{theo:blrate1} for all $p$ satisfying \eqref{equ:rangep}, we need two following techniques. The first one is an interpolation result from Cazenave and Lions \cite{CLcpde84}:
\begin{lemm}[\textbf{Interpolation technique, Cazenave and Lions \cite{CLcpde84}}] \label{lemm:intpola}  Assume that 
$$v \in L^\alpha\left( (0,\infty); L^\beta(\mathbf{B}_R) \right), \; v_t \in L^\gamma\left( (0,\infty); L^\delta(\mathbf{B}_R) \right)$$
for some $1 < \alpha, \beta, \gamma, \delta < \infty$. Then 
$$v \in \mathcal{C}\left([0,\infty); L^\lambda(\mathbf{B}_R) \right)$$
for all $\lambda  < \lambda_0 = \frac{(\alpha + \gamma')\beta\delta}{\gamma'\beta + \alpha\delta}$ with  $\gamma' = \frac{\gamma}{\gamma - 1}$, and satisfies
$$\sup_{t \geq 0} \|v(t) \| _{L^\lambda(\mathbf{B}_R)} \leq C \int_0^\infty \left(\|v(\tau)\|_{L^\beta(\mathbf{B}_R)}^\alpha + \|v_\tau(\tau)\|_{L^\delta(\mathbf{B}_R)}^\gamma \right)d\tau$$
for $\lambda < \lambda_0$. The positive constant $C$ depends only on $\alpha, \beta, \gamma, \delta , n$ and $R$.
\end{lemm}

\noindent The second one is an interior regularity result for a nonlinear parabolic equation:
\begin{prop}[\textbf{Interior regularity}] \label{prop:regpar} Let $v(x,t)\in L^\infty\big((0,+\infty), L^2(\mathbf{B}_R)\big) \cap L^2\big((0,+\infty),H^1(\mathbf{B}_R)\big)$ which satisfies  
\begin{equation}\label{equ:vinger}
v_t - \Delta v + b. \nabla v = F,\quad (x,t) \in Q_R = \mathbf{B}_R \times (0, +\infty),
\end{equation}
where $R > 0$, $|b(x,t)| \leq \mu_1$ in $Q_R$  and $|F(x,t,v)| \leq g(x,t)(|v| + 1)$ with
\begin{equation}\label{equ:conLemint}
\int_{t}^{t +1} \left\|g(\tau)\right\|^{\beta'}_{L^{\alpha'}(\mathbf{B}_R)}d\tau  \leq \mu_2, \quad \forall t \in(0, +\infty),
\end{equation}
and $\frac{1}{\beta'} + \frac{n}{2\alpha'} < 1$ and $\alpha' \geq 1$. If 
\begin{equation}\label{eq:conLeminreg1}
\int_{t}^{t +1} \|v(\tau)\|^2_{L^2(\mathbf{B}_R)} d\tau \leq \mu_3,\quad \forall t \in(0, +\infty),
\end{equation}
and $\mu_1$, $\mu_2$ and $\mu_3$ are uniformly bounded in $t$, then there exists a positive constant $C$ depending only on $\mu_1$, $\mu_2$, $\mu_3$, $\alpha'$, $\beta'$, $n$, $R$ and $\tau \in (0,1)$ such that
$$|v(x,t)| \leq C,\quad \forall(x,t) \in \mathbf{B}_{R/4} \times (\tau, +\infty).$$
\end{prop}
\begin{proof} Since the argument of the proof is analogous as in the corresponding part in \cite{LSU68}, we then leave the proof to Appendix \ref{ap:proRepa}.
\end{proof}
Let us now use Lemma \ref{lemm:keyint} to derive the conclusion of Theorem \ref{theo:blrate1}, then we will prove it later.
\begin{proof}[\textbf{Proof of Theorem \ref{theo:blrate1} admitting Lemma \ref{lemm:keyint}}] Let us recall the equation in $w$: 
\begin{equation*}
w_s - \Delta w + \frac{1}{2}y. \nabla w = -\frac{w}{p-1} + |w|^{p-1}w + e^{-\frac{ps}{p-1}}h\left(e^\frac{s}{p-1}w\right),
\end{equation*}
where $h$ is given in \eqref{equ:h}.\\
We now apply Proposition \ref{prop:regpar} to $w$ with $b = \frac{y}{2}$ and 
$$F = -\frac{w}{p-1} + |w|^{p-1}w + e^{-\frac{ps}{p-1}}h\left(e^\frac{s}{p-1}w\right).$$
From \eqref{equ:estimateH}, we see that
$$|F| \leq C'(C_0,p)(|w|^{p-1} + 1)(|w| + 1) , \quad \forall s \geq s_0.$$
Thus, the first identity in \eqref{equ:conLemint} holds with $g = C'(|w|^{p-1} + 1)$ and the second condition in \eqref{equ:conLemint} turns into 
$$\int_s^{s+1} \left(\int_{\mathbf{B}_R} |w(y,\tau)|^{\alpha'(p-1)}dy\right)^{\frac{\beta'}{\alpha'}}d\tau \leq C_1 \quad \text{for some $C_1 > 0$,}$$
for some $\alpha'$ and $\beta'$ satisfying $\frac{1}{\beta'} + \frac{n}{2\alpha'} < 1$.\\
For this bound, we first use $(i)$ of Proposition \ref{prop:boundEpsi}, \eqref{equ:keyint} and apply Lemma \ref{lemm:intpola} with $\alpha = q(p+1)$, $\beta = p+1$, $\gamma = \delta = \gamma' = 2$ to get that
\begin{equation}\label{equ:tmp2key}
\sup_{s \geq\hat{s}_2} \|w(s)\|_{L^\lambda(\mathbf{B}_R)} \leq C_2(R,K_q), \quad  \forall \lambda < \lambda_1 = p+1 - \frac{p-1}{q + 1}.
\end{equation}
Next, applying Proposition \ref{prop:regpar} with $\alpha'(p-1) = \lambda$, $\beta'$ and $q$ large (note that the condition $\frac{1}{\beta'} + \frac{n}{2\alpha'} < 1$ turns into $p  < \frac{n + 2}{n - 2}$), we obtain
$$\int_s^{s+1} \left(\int_{\mathbf{B}_R} |w(y,\tau)|^{\alpha'(p-1)}dy\right)^{\frac{\beta'}{\alpha'}}d\tau \leq C_2^{\beta'(p-1)}.$$
Hence, condition \eqref{equ:conLemint} holds. Therefore, $|w(y,s)|$ is bounded for all $(y,s) \in \mathbf{B}_{R/4} \times (\tau + \hat{s}_2, +\infty)$ for some $\tau \in (0,1)$, which concludes the proof of Theorem \ref{theo:blrate1}, assuming that Lemma \ref{lemm:keyint} holds.
\end{proof}

\begin{rema} If we use $(v)$ of Proposition \ref{prop:boundEpsi}, we already have for all $s \geq \tilde{s}_3$,
$$\int_s^{s+1}\left(\int_{\mathbf{B}_R} |w(y,\tau)|^{p+1}dy\right)^{2}d\tau \leq C(R)K_1.$$
Applying Proposition \ref{prop:regpar}  with $\alpha' = \frac{p+1}{p-1}$ and $\frac{\beta'}{\alpha'} = 2$ (noting that the condition $\frac{1}{\beta'} + \frac{n}{2\alpha'} < 1$ turns into $p < \frac{n+3}{n-1}$), we obtain $w$ is uniformly bounded with $ p \in\left(1, \frac{n+3}{n-1}\right)$.\\
\noindent If we use $(i)$ and $(v)$ in Proposition \ref{prop:boundEpsi}, Lemma \ref{lemm:intpola} with $\alpha = 2(p+1),\quad \beta = p+1, \quad \gamma = \delta = \gamma' = 2$, then we obtain 
$$\sup_{s \geq \tilde{s}_3} \|w(s)\|_{L^\lambda(\mathbf{B}_R)} \leq C(R), \quad \forall  \lambda < \lambda_1 = \frac{2(p+2)}{3}$$
Next, Proposition \ref{prop:regpar} applies with $\alpha'(p-1) = \lambda$ with $\lambda$ approaches to $\frac{2(p+1)}{3}$ and $\beta'$ very large, then the condition $\frac{1}{\beta'} + \frac{n}{2\alpha'} < 1$ now becomes 
$$\exists \lambda < \frac{2(p+1)}{3}, \quad \text{such that} \quad \frac{n}{2 \alpha'} < 1.$$
This turns into $p < \frac{3n + 8}{3n - 4}$. This result was proved by Giga and Kohn in \cite{GKiumj87}.\\
\noindent Relying on a bootstrap argument, \cite{GMSiumj04} improved the input estimate of Proposition \ref{prop:regpar} covering this way the whole subcritical range $p < \frac{n+2}{n-2}$. Here, we extend their approach to a larger class of equation.
\end{rema}

\vspace*{0.5cm}
Let us now give the proof of Lemma \ref{lemm:keyint} in order to complete the proof of Theorem \ref{theo:blrate1} and Theorem \ref{theo:blrate} also. To this end, let $\psi \in \mathcal{C}^2(\mathbb{R}^n)$ be a bounded function, we introduce the following local functional, which is a perturbed version of the function of \cite{GMSiumj04},
\begin{align}
\mathcal{E}_\psi[w](s) & = \frac{1}{2} \int_{\mathbb{R}^n}\psi^2 \left(|\nabla w|^2 + \frac{1}{p-1}|w|^2 \right)\rho dy \nonumber\\
&\quad  -\frac{1}{p+1}\int_{\mathbb{R}^n}\psi^2 |w|^{p+1}\rho dy - e^{-\frac{p+1}{p-1}s} \int_{\mathbb{R}^n}\psi^2 H\left(e^{\frac{s}{p-1}}w \right) \rho dy. \label{equ:localE}
\end{align}
We get the following bound on the local functional $\mathcal{E}_\psi$:
\begin{prop} \label{prop:upElc} Let $a, p, n, M$ be fixed and $w$ be solution of equation \eqref{equ:divw1} satisfying \eqref{equ:proW}. For $\psi \in \mathcal{C}^2(\mathbb{R}^n)$ bounded, there exist positive constants $Q', K'$ such that 
\begin{equation}\label{equ:upElc}
-Q' \leq \mathcal{E}_\psi[w](s) \leq K', \quad \forall s \geq \tilde{s}_3,
\end{equation}
where $\tilde{s}_3$ is given in Proposition \ref{prop:boundEpsi} and 
$Q', K'$ depend on $a$, $p$, $n$, $M$, $\|\psi\|^2_{L^\infty}$, $\|\nabla \psi\|^2_{L^\infty}$ and $J_0$.
\end{prop}

\begin{proof} The proof is essentially the same as the corresponding part in \cite{GMSiumj04}, except for the control of the last term in \eqref{equ:localE}. Since that control is a bit long and technical, we leave the proof to \ref{ap:upELc}.
\end{proof}
\vspace*{0.3cm}
\noindent Let $R > 0$, we fix $\psi(y)$ so that it satisfies 
\begin{equation} \label{equ:psiy}
\psi(y) \in \mathcal{C}_0^\infty(\mathbb{R}^n), \quad 0 \leq \psi(y) \leq 1, \quad \psi(y) = \left\{\begin{array}{lcl} 1 & \quad\text{on} &\quad \mathbf{B}_R \\
0 & \quad \text{on} & \quad \mathbb{R}^n\setminus \mathbf{B}_{2R} \end{array} \right.. 
\end{equation}
We claim the following:
\begin{lemm}\label{rema:boundLpW12} Let $a, p, n, M$ be fixed and $w$ be solution of equation \eqref{equ:divw1} satisfying \eqref{equ:proW}. Then there exists $\tilde{s}_5 \geq \tilde{s}_3$ such that 
\begin{equation}\label{equ:boundLp1}
\|w\|_{L^{p+1}_\rho(\mathbf{B}_R)}^{p+1} \leq K_1\left(1 + \|w \|^2_{H^1_\rho(\mathbf{B}_{2R})}\right),\quad \forall s \geq \tilde{s}_5,
\end{equation}
where $K_1 = K_1(a,p, n, M,Q')$.
\end{lemm}
\begin{proof} From \eqref{equ:estimateHhe} and the definition of $\mathcal{E}_\psi$ in \eqref{equ:localE}, we have $\forall s \geq \max\{s_0,s_1\}$,
\begin{align}
\int_{\mathbb{R}^n}\psi^2 |w|^{p+1}\rho dy &\leq -2(p+1)\mathcal{E}_\psi[w](s)\nonumber\\
&+ (p+1)\int_{\mathbb{R}^n} \psi^2 \left(|\nabla w|^2 + \frac{1}{p-1}|w|^2 \right)\rho dy + 1,\label{eq:tmpEu12}
\end{align}
where $s_1$ is large enough such that $2C_0s^{-a} \leq \frac{1}{2(p+1)}$ for all $s \geq s_1$.\\
Thus, \eqref{equ:boundLp1} follows from the lower bound of $\mathcal{E}_\psi$ and the property of $\psi$. This ends the proof of Lemma \ref{rema:boundLpW12}.
\end{proof}
\begin{rema} \label{rema:conctLW}
By \eqref{equ:boundLp1}, the proof of estimate \eqref{equ:keyint} is equivalent to showing that
\begin{equation}\label{equ:keyintW}
\int_s^{s+1} \|w(\tau)\|_{H^1_\rho(\mathbf{B}_{R})}^{2q} d\tau \leq K_q, \quad\forall s \geq \hat{s}_2.
\end{equation}
Note from $(i)$ and $(iv)$ in Proposition \ref{prop:boundEpsi} that \eqref{equ:keyintW} already holds in the case $q = 2$.
\end{rema}

\noindent In order to derive \eqref{equ:keyintW} for all $q \geq 2$, we need the following result:
\begin{lemm} \label{lemm:tmpq2} Let $a, p, n, M$ be fixed and $w$ be solution of equation \eqref{equ:divw1} satisfying \eqref{equ:proW}. Then there exists $\tilde{s}_6 \geq \tilde{s}_3$ such that 
\begin{equation}\label{equ:boundLp123}
\|w\|_{H^1_\rho(\mathbf{B}_R)}^{2} \leq K_2\left(1 + \|\psi^2 w w_s \|^2_{L^1_\rho(\mathbf{B}_{2R})}\right),\quad \forall s \geq \tilde{s}_6,
\end{equation}
where $K_2 = K_2(a,p, n, M,Q',K')$.
\end{lemm}
\begin{proof}
Multiplying equation \eqref{equ:divw1} with $\psi^2w\rho$, integrating over $\mathbb{R}^n$, using the definition of $\mathcal{E}_\psi$ and estimate  \eqref{equ:estimateHhe}, we have
\begin{align*}
\int_{\mathbb{R}^n} \psi^2 \left(|\nabla w|^2 + \frac{1}{p-1}|w|^2 \right)\rho dy & \leq \frac{2}{p-1}\int_{\mathbb{R}^n} \psi^2 w w_s \rho dy + \frac{2(p+1)}{p-1} \mathcal{E}_\psi[w](s)\\
& + \frac{4}{p-1}\int_{\mathbb{R}^n} \psi w \nabla \psi. \nabla w \rho dy \\
& + \frac{4(p+1)C_0}{(p-1)s^a}\int_{\mathbb{R}^n}\psi^2(|w|^{p+1} + 1)\rho dy, \;\; \forall s \geq s_0.
\end{align*}
Using \eqref{eq:tmpEu12}, then taking $s_2$ large such that $\frac{4(p+1)^2C_0}{(p-1)s^a}\leq \frac{1}{2}$ and noting that $\mathcal{E}$ is bounded, we have for all $s \geq \max\{s_0,s_1,s_2\}$,
$$\int_{\mathbb{R}^n} \psi^2 \left(|\nabla w|^2 + \frac{1}{p-1}|w|^2 \right)\rho dy \leq C\left(\int_{\mathbb{R}^n} \psi^2 w w_s \rho dy +  \int_{\mathbb{R}^n} \psi w \nabla \psi. \nabla w \rho dy + 1\right).$$
Let $J_\psi[w](s) = \int_{\mathbb{R}^n} \psi w \nabla \psi. \nabla w \rho dy$, then one can show that $J_\psi[w](s) \leq C_1$ (see \eqref{equ:estJws} for a proof of this fact). Hence, we have for all $s \geq \max\{s_0,s_1,s_2\}$,
$$\int_{\mathbb{R}^n} \psi^2 \left(|\nabla w|^2 + \frac{1}{p-1}|w|^2 \right)\rho dy \leq C_2\left(\int_{\mathbb{R}^n} \psi^2 w w_s \rho dy + 1\right).$$
Thus, \eqref{equ:boundLp123} follows from the property of $\psi$, and Lemma \ref{lemm:tmpq2} is proved.
\end{proof}

\noindent Since the estimate \eqref{equ:keyintW} already holds in the case $q = 2$, we now use a bootstrap argument in order to get \eqref{equ:keyintW} for all $q \geq 2$.
\begin{proof}[\textbf{Proof of \eqref{equ:keyintW} for all $q \geq 2 $ by a bootstrap argument}] This part is the same as in \cite{GMSiumj04}. We give it here for the sake of completeness. Suppose that \eqref{equ:keyintW} holds for some $q \geq 2$, let us show that \eqref{equ:keyintW} holds for all  $\tilde{q} \in [q, q + \epsilon]$ for some $\epsilon > 0$ independent from $q$. We start with Holder's inequality,
\begin{equation*}
\|\psi^2 w w_s \|_{L^1_\rho(\mathbf{B}_{2R})} \leq \|\psi w \|_{L^\lambda_\rho(\mathbf{B}_{2R})} \times \|\psi w_s \|_{L^{\lambda'}_\rho(\mathbf{B}_{2R})}, \quad \frac{1}{\lambda} + \frac{1}{\lambda'} = 1.
\end{equation*}
Using \eqref{equ:keyint} and applying Lemma \ref{lemm:intpola}, we obtain
$$\|w\|_{L^\lambda(\mathbf{B}_{2R})} \leq C_q', \quad \forall \lambda < \lambda_1(q) = p+1  - \frac{p-1}{q+1}.$$
Let us now bound $\|\psi w_s \|_{L^{\lambda'}_\rho(\mathbf{B}_{2R})}$. We remark that for $q$ large then $\lambda$ approaches to $p+1$ and $\lambda'$ approaches to $p_1 = \frac{p+1}{p}$. Let $f = \psi w_s$ and make use Holder's inequality,
$$\|f\|_{L^{\lambda'}} \leq \|f\|_{L^2}^{1 - \theta} \times \|f\|_{L^{p_1}}^\theta, \quad  \frac{1}{\lambda'} = \frac{1 - \theta}{2} + \frac{\theta}{p_1}, \quad \theta \in [0, 1].$$
From now on, we take $\lambda \geq 2$ and fix $\theta = \frac{(\lambda - 2)(p+1)}{\lambda(p-1)}$ (note that with this choice, $\theta \in [0,1]$). From Lemma \ref{lemm:tmpq2}, we have
\begin{equation*}
\|w(s)\|_{H^1_\rho(\mathbf{B}_R)}^2 \leq K_2'\left(1 +  \|\psi w_s\|_{L^2_\rho(\mathbf{B}_{2R})}^{1 - \theta} \times \|\psi w_s\|_{L^{p_1}_\rho(\mathbf{B}_{2R})}^\theta \right).
\end{equation*}
This follows that
\begin{equation}
\int_s^{s+1} \|w(s)\|_{H^1_\rho(\mathbf{B}_R)}^{2\tilde{q}} d\tau \leq C_{\tilde{q}}\left[1 + \underbrace{\int_s^{s+1} \|\psi w_s\|_{L^2_\rho(\mathbf{B}_{2R})}^{\tilde{q}(1 - \theta)} \times \|\psi w_s\|_{L^{p_1}_\rho(\mathbf{B}_{2R})}^{\tilde{q}\theta} d\tau}_{\mathbf{G}}\right],
\end{equation}
for some $\tilde{q} > q$.\\
Let $\alpha = \frac{2}{(1 - \theta)\tilde{q}}$ and use Holder's inequality in time to $\mathbf{G}$, we obtain
\begin{align*}
\mathbf{G}  &\leq \left(\int_s^{s+1}\|\psi w_s\|_{L^2_\rho(\mathbf{B}_{2R})}^2 d\tau\right)^{\frac{1}{\alpha}}
\left(\int_s^{s+1}\|\psi w_s\|_{L^{p_1}_\rho(\mathbf{B}_{2R})}^{\tilde{q}\theta \alpha'} d\tau\right)^{\frac{1}{\alpha'}}\\
&\leq (2J_0)^\frac{1}{\alpha}\left(\int_s^{s+1}\|\psi w_s\|_{L^{p_1}_\rho(\mathbf{B}_{2R})}^{\tilde{q}\theta \alpha'} d\tau\right)^{\frac{1}{\alpha'}} \equiv \mathbf{G}_1,
\end{align*}
where we used $(i)$ in Proposition \ref{prop:boundEpsi}.\\
Let us bound $\mathbf{G}_1$. To this end, we use the $L^p-L^q$ estimate for the heat equation (see Lemmas 6.3 and 6.4 in \cite{GMSiumj04}) to get
\begin{align*}
\int_s^{s+1}\|\psi w_s\|_{L^{p_1}_\rho(\mathbf{B}_{2R})}^{\tilde{q}\theta \alpha'}d\tau &\leq C_{\tilde{q}}' \left(1 +  \int_s^{s+1}\left\||w|^p\right\|_{L^{p_1}_\rho(\mathbf{B}_{2R})}^{\tilde{q}\theta \alpha'}d\tau\right)\\
& = C_{\tilde{q}}' \left(1 +  \int_s^{s+1}\left\|w\right\|_{L^{p+1}_\rho(\mathbf{B}_{2R})}^{p\tilde{q}\theta \alpha'}d\tau\right)\\
& \leq C_{\tilde{q}}'' \left(1 +  \int_s^{s+1}\left\|w\right\|_{H^1_\rho(\mathbf{B}_{4R})}^{\frac{2p\tilde{q}\theta \alpha'}{p+1}}d\tau\right) \quad \text{(using Lemma \ref{eq:tmpEu12}).} 
\end{align*}
By Proposition 6.2 in \cite{GMSiumj04}, we have $\frac{2p\tilde{q}\theta \alpha'}{p+1} < 2q$ for all $\tilde{q} \in [q, q + \frac{2}{p+1}]$. Then, applying Holder's inequality again yields
\begin{align*}
\int_s^{s+1} \|w(s)\|_{H^1_\rho(\mathbf{B}_R)}^{2\tilde{q}} d\tau \leq C_{\tilde{q}}'''\left[1 + \left(\int_s^{s+1} \|w(s)\|_{H^1_\rho(\mathbf{B}_{4R})}^{2q} d\tau\right)^\frac{1}{2q \alpha'} \right] \leq \bar{C}_{\tilde{q}}.
\end{align*}
Thus, inequality \eqref{equ:keyintW} is valid for all $\tilde{q} \in [q, q + \frac{2}{p+1}]$. Repeating this argument, we would obtain that \eqref{equ:keyintW} holds for all $q \geq 2$. This concludes the proof of Lemma \ref{lemm:keyint}, Theorem \ref{theo:blrate1} and Theorem \ref{theo:blrate} too.
\end{proof}

\subsection{Limit of {$w$} as {$s \to +\infty$}} \label{sec:limit}
This section is devoted to the proof of Theorem \ref{theo:wtoin}. Note in the unperturbed case ($h \equiv 0$) that Theorem \ref{theo:wtoin} was proved in \cite{GKcpam89} (see also \cite{GKcpam85}, \cite{GKiumj87}). The proof is divided into two steps. The first step is to show that the limit of solution in similarity variables exists and belongs to the set of solutions of the following equation, 
\begin{equation}\label{equ:staSol}
0 = \Delta w - \frac{1}{2}y.\nabla w - \frac{1}{p-1}w + |w|^{p-1}w,
\end{equation}
Then, by using a nondegeneracy result (Lemma \ref{lemm:nonde}), the blow-up criterion (Lemma \ref{lemm:cri}) and suitable energy arguments, we shall show that the possibility of $w_a \to 0$ as $s \to +\infty$ is excluded if $a$ is a blow-up point. Let us restate Theorem \ref{theo:wtoin} in below:
\begin{prop}[\textbf{Limit of $w$ as $s \to +\infty$}] \label{prop:wtoin}
Let $a, p, n, M$ be fixed, $p$ be a sub-critical non-linearity given in \eqref{equ:rangep}. Consider $u(t)$ a solution of equation \eqref{equ:problem} which blows up at time $T$ and $a$ a blow-up point. Then\\
$$\lim_{s \to +\infty} \,w_{a}(y,s) = \pm \kappa, \quad \text{uniformly on each compact subset of $\mathbb{R}^n$.}$$
\end{prop}
Before going into the proof of Proposition \ref{prop:wtoin}, let us first derive some elementary results. The first one concerns the stationary solutions in $\mathbb{R}^n$ of equation \eqref{equ:staSol}. Particularly, we have the following:
\begin{lemm}[\textbf{Stationary solutions, Giga and Kohn \cite{GKcpam85}}]\label{lemm:stasol}
Let $p$ satisfy \eqref{equ:rangep}, then all bounded solutions of \eqref{equ:staSol} are constants: $w \equiv 0$ or $w \equiv \pm \kappa$.
\end{lemm}
\begin{proof} The proof is given in Proposition 2 of \cite{GKcpam85}. For the reader's interest, we mention that the proof relies on a clever use of multiplying factors, together with a Pohozaev technique, resulting in the following identity:
\begin{equation}\label{equ:Pohoden}
\left(\frac{n}{p+1} - \frac{2-n}{2}\right)\int_{\mathbb{R}^n}|\nabla w|^2\rho dy + \frac{1}{2}\left(\frac{1}{2} - \frac{1}{p+1}\right)\int_{\mathbb{R}^n}|y|^2|\nabla w|^2 \rho dy = 0.
\end{equation}
From \eqref{equ:Pohoden} and the fact that $p$ is Sopolev subcritical, it follows that $\frac{n}{p+1} - \frac{2-n}{2} > 0$ and $\frac{1}{2} - \frac{1}{p+1} > 0$, hence $\nabla w \equiv 0$. This implies that $w$ is actually a constant.  This concludes the proof of Lemma \ref{lemm:stasol}.
\end{proof}
The second one is due to parabolic estimates:
\begin{lemm}[\textbf{Parabolic estimates}]\label{lemm:regupa} Let $u$ be a solution to equation \eqref{equ:problem}. Assume that $T= T_{\max}(u_0)  < +\infty$ and that $u$ satisfies \eqref{equ:blrateu1}. Then, there is a positive constant $C$ such that for all $t \in [T/2, T)$,
\begin{equation}\label{equ:parestUt}
\|\nabla u(t)\|_{L^\infty(\mathbb{R}^n)} \leq C(T-t)^{-\frac{1}{p-1} - \frac{1}{2}} \quad \text{and} \quad \|\nabla^2 u(t)\|_{L^\infty(\mathbb{R}^n)} \leq C(T-t)^{-\frac{1}{p-1} - 1}.
\end{equation}
In similarity variables, we have for all $s \in [-\log(T/2), +\infty)$ and $x_0 \in \mathbb{R}^n$,
\begin{equation}\label{equ:paraEst}
\|\nabla w_{x_0,T}(s)\|_{L^\infty(\mathbb{R}^n)} \leq C \quad \text{and} \quad \|\nabla^2 w_{x_0,T}(s)\|_{L^\infty(\mathbb{R}^n)} \leq C. 
\end{equation}
\end{lemm}
\begin{proof} Since $|h(z)| \leq C(|z|^p + 1)$ and $|h'(z)| \leq C(|z|^{p-1} + 1)$ from \eqref{equ:h}, the proof given in Proposition 23.15, page 189 of Souplet and Quittner \cite{QPbook07} in the case $h \equiv 0$ extends with no difficulty in this case.
\end{proof}
The last one is the nondegeneracy result from Giga and Kohn \cite{GKcpam89}:
\begin{lemm}[\textbf{Nondegeneracy, Giga and Kohn \cite{GKcpam89}}]\label{lemm:nonde} Let $p > 1$, $T > 0$, $r > 0$, $\sigma \in (0,1)$, $a \in \mathbb{R}^n$ and denote $Q_{r,\sigma}(a) = \mathbf{B}_r(a) \times (T - \sigma, T)$. There exists $\epsilon = \epsilon(n,p) > 0$ such that if $u$ is a classical solution of 
\begin{equation}\label{equ:problmeLi}
u_t - \Delta u = F(u), \quad (x,t) \in Q_{r,\sigma}(a),
\end{equation}
where $|F(u)| \leq M(|u|^p + 1)$ for some $M > 0$. Assume that $u$ satisfies 
\begin{equation}\label{equ:Noncond}
|u(x,t)| \leq \epsilon(T-t)^{-\frac{1}{p-1}}, \quad (x,t) \in Q_{r,\sigma}(a),
\end{equation}
then $u$ is uniformly bounded in a neighborhood of $(a,T)$.
\end{lemm}
\begin{proof} See Theorem 2.1, page 850 in Giga and Kohn \cite{GKcpam89}.
\end{proof}

Let us now give the proof of Proposition \ref{prop:wtoin}.
\begin{proof}[\textbf{Proof of Proposition \ref{prop:wtoin}}] Consider $a$ a blow-up point and write $w$ instead of $w_{a}$ for simplicity. By Lemma \ref{lemm:regupa} and equation \eqref{equ:divw1}, we see that $|w_s(y,s)| \leq C(|y| + 1)$ for some $C>0$. Therefore, $w$, $\nabla w$, $\nabla^2w$ and $w_s$ are bounded for all $|y| \leq R$ and $s \geq s'$ for some $R>0$ and $s' \in \mathbb{R}$. Let $\{s_j\}$ be a sequence tending to $+\infty$ and  $w_j(y,s) = w(y, s+ s_j)$. By the Arzela-Ascoli theorem, there is a subsequence of $s_j$ (still denoted $s_j$) such that $w_j$ converges uniformly on compact sets to some $w^{\infty}$, $\nabla w_{j} \to \nabla w^{\infty}$, $\Delta w_j \to \Delta w^\infty$ and $w_{js} \to w^{\infty}_s$. On the other hand, by $(i)$ and $(vi)$ of Proposition \ref{prop:boundEpsi}, we see that as $j \to +\infty$,
\begin{align*}
&\int_{\tilde{s}_3}^{+\infty}\int_{\mathbf{B}_R} |w_{js}|^2 dy ds = \int_{\tilde{s}_3 + s_j}^{+\infty} \int_{\mathbf{B}_R} |w_s|^2 dy ds \to 0.
\end{align*}
This implies that $w^{\infty}_s = 0$ and $w^\infty$ satisfies \eqref{equ:staSol}. Hence, by Lemma \ref{lemm:stasol}, $w^{\infty} \equiv 0$ or $w^{\infty} \equiv \pm\kappa$. \\
It remains to show that $w(\cdot,s_j) \nrightarrow 0$ as $j \to +\infty$. We proceed by contradiction. Let us  assume that $w(\cdot, s_j) \to 0$ as $j \to +\infty$. We observer that if $w(\cdot, s_j) \to 0$, then by the definition of $\mathcal{J}$ given in \eqref{equ:lyafun}, the bound of $w$ and $\nabla w$ and dominated convergence, then $\mathcal{J}[w](s_j) \to 0$. Since $\mathcal{J}$ is a Lyapunov functional, it follows that the whole sequence 
\begin{equation}\label{eq:Jtendto0}
\mathcal{J}[w](s) \to 0 \quad \text{as}\quad s \to +\infty.
\end{equation}
Let $b \in \mathbb{R}^n$, then by \eqref{equ:paraEst}, we have $w_b(y,s)$ and $\nabla w_b(y,s)$ are bounded for all $y \in \mathbb{R}^n$ and $s \geq s'$. We now use the interpolation inequality which reads
$$|w_b(0,s)| \leq C\left(\|w_b\|_{L^2(\mathbf{B}_R)}^\theta\|\nabla w_b\|^{1-\theta}_{L^\infty(\mathbf{B}_R)} + \|w_b\|_{L^2{(\mathbf{B}_R)}}\right),$$
where $\theta \in (0, \frac{2}{n+2})$ if $n \geq 2$ and $\theta = 1/2$ if $n = 1$.\\
By Lemma \ref{lemm:cri}, we see that $\|w_b(s)\|_{L^2(\mathbf{B}_R)} \leq C(p)\big(\mathcal{J}[w_b](s)\big)^{\frac{1}{p+1}}$ for all $s \geq \tilde{s}_1$. Hence, 
$$|w_b(0,s)| \leq C' \left(\big(\mathcal{J}[w_b](\tilde{s}_1)\big)^{\frac{\theta}{p+1}} + \big(\mathcal{J}[w_b](\tilde{s}_1)\big)^{\frac{1}{p+1}}\right), \quad \forall s \geq \tilde{s}_1.$$
Consider some $\epsilon > 0$ small. From \eqref{eq:Jtendto0}, there is $s'(\epsilon)$ such that  $\mathcal{J}[w](s) \leq \epsilon $ for all $s \geq s'(\epsilon)$. Therefore, by continuity depending of $\mathcal{J}[w_b](s)$ on $b$ and the monotonicity of $\mathcal{J}[w_b](s)$ in time $s$, we infer that $\mathcal{J}[w_b](s) \leq 2\epsilon$ for all $s \geq s'$ and $|b-a|$ small. This implies that $|w_b(0,s)| \leq \epsilon''$ for all $s \geq s'$, or $|u(b,t)| \leq \epsilon''(T-t)^{-\frac{1}{p-1}}$ for $(b,t)$ close to $(a,T)$, where $\epsilon'' = \epsilon''(\epsilon) \to 0$ as $\epsilon \to 0$. Thus, $a$ is not a blow-up point by Lemma \ref{lemm:nonde}, and this is a contradiction. Therefore, this concludes the proof of Proposition \ref{prop:wtoin} and the proof of Theorem \ref{theo:wtoin} also.
\end{proof}

\section{Classification of the behavior of $w$ as $s \to +\infty$ in $L_\rho^2$} \label{sec:refasy}
This section is devoted to the proof of Theorem \ref{theo:refinedasymptotic}. Consider $a$ a blow-up point and write $w$ instead of $w_a$ for simplicity. From Theorem \ref{theo:wtoin} and up to changing the signs of $w$ and $h$, we may assume that 
$\|w(y,s) -\kappa\|_{L^2_\rho} \to 0$ as $s \to +\infty$, uniformly on compact subsets of $\mathbb{R}^n$. As mentioned in the introduction, by setting $v(y,s) = w(y,s) - \phi(s)$ ($\phi$ is a positive solution of \eqref{equ:phiODE} such that $\phi(s) \to \kappa$ as $s \to + \infty$), we see that $\|v(y,s)\|_{L^2_\rho} \to 0$ as $s \to + \infty$ and $v$ solves the following equation: 
\begin{equation}\label{equ:v}
\partial_s v = (\mathcal{L} + \omega(s))v + F(v) + H(v,s), \quad \forall y \in \mathbb{R}^n,\; \forall s \in [-\log T, +\infty),
\end{equation}
where $\mathcal{L} = \Delta - \frac{y}{2}\cdot \nabla + 1$ and $\omega$, $F$, $H$ are given by
\begin{align*}
&\omega(s) = p\left(\phi^{p-1} - \kappa^{p-1}\right)+ e^{-s}h'\left(e^\frac{s}{p-1}\phi\right),\\
&F(v) = |v+ \phi |^{p-1}(v+\phi) - \phi^p - p\phi^{p-1}v, \\
&H(v,s) = e^{-\frac{ps}{p-1}}\left[h\left(e^\frac{s}{p-1}(v+\phi)\right)-h\left(e^\frac{s}{p-1}\phi\right) - e^\frac{s}{p-1}h'\left(e^\frac{s}{p-1}\phi\right)v\right].
\end{align*}
We remark from \eqref{equ:solphi}  and \eqref{equ:estimateH} that
\begin{equation}\label{equ:omega}
|\omega(s)| = \mathcal{O}\left(\frac{1}{s^a}\right) \quad \text{as} \quad s \to +\infty.
\end{equation}
Let us introduce for all $y \in \mathbb{R}^n$, for all $s \in [-\log T, +\infty)$,
\begin{equation}\label{def:V}
\beta(s) = e^{-\int_s^{+\infty}\omega(\tau)d\tau} \quad \text{and} \quad V(y,s) = \beta(s)v(y,s),
\end{equation} 
(note that $\beta(s) \to 1$ as $s \to +\infty$).\\
By multiplying equation \eqref{equ:v} to $\beta(s)$, we find the following equation satisfied by $V$:
\begin{equation}\label{equ:V}
\partial_s V = \mathcal{L}V + \bar{F}(V,s), \quad \forall y \in \mathbb{R}^n,\; \forall s \in [-\log T, +\infty),
\end{equation}
where $\bar{F}(V,s) = \beta(s)(F(v) + H(v,s))$ satisfying 
\begin{equation}\label{equ:estFbar2}
|\bar{F}(V,s)| \leq CV^2.
\end{equation}
Since $\|w(s)\|_{L^\infty} \leq C$ from Theorem \ref{theo:blrate}, we may use a Taylor expansion, \eqref{equ:estimateH}, \eqref{equ:solphi} and the fact that $\beta(s) = 1 + \mathcal{O}\left(\frac{1}{s^{a-1}}\right)$ as $s \to +\infty$ to write
\begin{equation}\label{rem:boundF_H}
\left|\bar{F}(V,s) - \frac{p}{2\kappa}V^2 \right| = \mathcal{O}(|V|^3) + \mathcal{O}\left(\frac{V^2}{s^{a-1}}\right) \quad \text{as $s \to +\infty$},
\end{equation}
(see Lemma \ref{ap:lemmbounF_H} for the proof of \eqref{rem:boundF_H}, and note that \eqref{equ:estFbar2} follows from \eqref{rem:boundF_H}).\\
\noindent Since the eigenfunctions of $\mathcal{L}$ constitute a total orthonormal family of $L^2_\rho$, we can expand $V$ as follows: 
\begin{equation}\label{equ:expv}
V(y,s) = \sum_{k = 1}^\infty \pi_k(V)(y,s) =  V_{+}(y,s) + V_{null}(y,s) + V_{-}(y,s),
\end{equation}
where $\pi_k(V)$ is the orthogonal projector of $v$ on the eigenspace associated to $\lambda_k = 1 - \frac{k}{2}$, 
\begin{align}
V_{+}(y,s) & = \pi_+(V)(y,s) = \sum_{k=0}^1\pi_k(V)(y,s),\nonumber\\
V_{-}(y,s) &= \pi_-(V)(y,s) = \sum_{k=3}^\infty\pi_{k}(V)(y,s),\nonumber\\
V_{null}(y,s) &=\pi_2(V)(y,s) = V_2(s)\centerdot H_2(y),\label{equ:defpi101}
\end{align}
where $H_2(y) = \left(H_{2,ij}, i \leq j\right)$, with $H_{2,ii} = h_2(y_i)$ and $H_{2,ij} = h_1(y_i)h_1(y_j)$ if $i \ne j$, $h_m$ is introduced in \eqref{equ:eigenfu}; $V_2(s) = \left(V_{2,ij}, i \leq j\right)$, with $V_{2,ij}$ being the projection of $V$ on $H_{2,ij}$.\\

\noindent We claim that Theorem \ref{theo:refinedasymptotic} is a direct consequence of the following:
\begin{prop}[\textbf{Classification of the behavior of $V$ as $s \to +\infty$}]\label{prop:refinedasymptotic} One of the following possibilities occurs:\\
$i)\;$ $V(y,s) \equiv 0$,\\
$ii)$ There exists $l \in \{1, \dots, n\}$ such that up to an orthogonal transformation of coordinates, we have
$$
V(y,s) = -\frac{\kappa}{4ps} \left(\sum_{j=1}^l y_j^2 - 2l\right) + \mathcal{O}\left(\frac{1}{s^a} \right) + \mathcal{O}\left(\frac{\log s}{s^2}\right) \quad \text{as} \quad s \to +\infty.
$$
$iii)$ There exist an integer number $m \geq 3$ and constants $c_\alpha$ not all zero such that 
$$V(y,s) = - e^{\left(1 - \frac{m}{2}\right)s} \sum_{|\alpha| = m}c_\alpha H_\alpha(y) + o\left(e^{\left(1 - \frac{m}{2}\right)s}\right)\quad \text{as} \quad s \to +\infty.$$
The convergence takes place in $L^2_\rho$ as well as in $\mathcal{C}^{k,\gamma}_{loc}$ for any $k \geq 1$ and $\gamma \in (0,1)$.
\end{prop}
\begin{rema}
Let us insist on the fact that the linearizing of $w$ around $\kappa$ would generate some terms of the size $\frac{1}{s^a}$, and prevent us from reaching exponentially small terms. \end{rema}
\noindent Let us first derive Theorem \ref{theo:refinedasymptotic} assuming Proposition \ref{prop:refinedasymptotic} and then we will prove it later.
\begin{proof}[\textbf{Proof of Theorem \ref{theo:refinedasymptotic} assuming that Proposition \ref{prop:refinedasymptotic} holds}] By the definition \eqref{def:V} of $V$, we see that $i)$ of Proposition \ref{prop:refinedasymptotic} directly follows that $v(y,s) \equiv \phi(s)$ which is $i)$ of Theorem \ref{theo:refinedasymptotic}. Using $ii)$ of Proposition \ref{prop:refinedasymptotic} and the fact that $\beta(s) = 1 + \mathcal{O}(\frac{1}{s^{a-1}})$ as $s \to +\infty$, we see that as $s \to +\infty$,
\begin{align*}
w(y,s) &= \phi(s) + V(y,s)\left(1 + \mathcal{O}(\frac{1}{s^{a-1}})\right)\\
&= \phi(s) -\frac{\kappa}{4ps} \left(\sum_{j=1}^l y_j^2 - 2l\right) + \mathcal{O}\left(\frac{1}{s^a} \right) + \mathcal{O}\left(\frac{\log s}{s^2}\right)\\
&=\kappa -\frac{\kappa}{4ps} \left(\sum_{j=1}^l y_j^2 - 2l\right) + \mathcal{O}\left(\frac{1}{s^a} \right) + \mathcal{O}\left(\frac{\log s}{s^2}\right),
\end{align*}
which yields $ii)$ of Theorem \ref{theo:refinedasymptotic}.\\
Using $iii)$ of Proposition \ref{prop:refinedasymptotic} and again the fact that $\beta(s) = 1 + \mathcal{O}(\frac{1}{s^{a-1}})$ as $s \to +\infty$, we have 
$$w(y,s) = \phi(s)- e^{\left(1 - \frac{m}{2}\right)s} \sum_{|\alpha| = m}c_\alpha H_\alpha(y) + o\left(e^{\left(1 - \frac{m}{2}\right)s}\right)\quad \text{as} \quad s \to +\infty.$$
This concludes the proof of Theorem \ref{theo:refinedasymptotic} assuming that Proposition \ref{prop:refinedasymptotic} holds.
\end{proof}
\noindent The proof of Proposition \ref{prop:refinedasymptotic} will be very close to that in \cite{FLaihn93} and \cite{VELtams93}, thanks to \eqref{equ:estFbar2} and \eqref{rem:boundF_H}. It happens that the proofs written in Filippas, Kohn, Liu, Herrero and Vel\'azquez \cite{FKcpam92},\cite{FLaihn93}, \cite{HVaihn93}, \cite{VELtams93} in the unperturbed case ($h \equiv 0$) hold for equation \eqref{equ:V} under the general assumptions \eqref{equ:estFbar2} and \eqref{rem:boundF_H}. For that reason, we only give the sketch of the proof below and refer to these papers for details of the proofs. \\
\noindent Following \cite{FLaihn93} and \cite{VELtams93}, we divide the proof into 3 steps which are given in separated subsections:\\
- Step 1: deriving the fact that either  $\|V_+(s)\|_{L^2_\rho} + \|V_-(s)\|_{L^2_\rho} = o\left(\|V_{null}(s)\|_{L^2_\rho} \right)$, or $\|V(s)\|_{L^2_\rho} = \mathcal{O}(e^{-\mu s})$ for some $\mu > 0$. \\
- Step 2: assuming that $\|V(y,s)\|_{L^2_\rho} \sim \|V_{null}(y,s)\|_{L^2_\rho}$, we find an equation satisfied by $V_{null}(s)$ as $s \to +\infty$. Solving this equation, we find that $\|V(s)\|_{L^2_\rho}$ behaves like $\frac{1}{s}$ as $s \to +\infty$. Using this information, we can get a more accurate equation for $V_{null}(s)$ as $s \to + \infty$ and then $ii)$ of Proposition \ref{prop:refinedasymptotic} follows.\\
- Step 3: assuming $\|V(s)\|_{L^2_\rho} = \mathcal{O}(e^{-\mu s})$ for some $\mu > 0$ as $s \to +\infty$, we derive $i)$ or $iii)$ of  Proposition \ref{prop:refinedasymptotic}.

\subsection{Finite dimension reduction of the problem.}
\noindent We claim the following proposition: 
\begin{prop}[\textbf{Competition between $V_+, V_-$ and $V_{null}$}] \label{prop:reducetovnull} As $s \to +\infty$,
\begin{align}
\text{either} \quad & i) \; \|V(s)\|_{L^2_\rho} = \mathcal{O}\left(e^{-\mu s}\right),\quad \text{for some $\mu > 0$,} \qquad \qquad \qquad\label{equ:expo}\\
\text{or} \quad & ii) \;\|V_{+}(s)\|_{L^2_\rho}+ \|V_{-}(s)\|_{L^2_\rho} = o\left(\|V_{null}(s)\|_{L^2_\rho}\right). \qquad \qquad \label{equ:reduce2vnull}
\end{align}
\end{prop}
\begin{proof}
Let us denote 
\begin{equation}\label{equ:difzxy}
Z(s) = \|V_{+}(s)\|_{L^2_\rho}, \quad X(s) = \|V_{null}(s)\|_{L^2_\rho}, \quad Y(s) = \|V_{-}(s)\|_{L^2_\rho},
\end{equation}
then the following lemma is claimed:
\begin{lemm} \label{lemm:ODEsys} 
Let $\epsilon > 0$, there exists $s^* = s^*(\epsilon) \in \mathbb{R}$ such that for all $s \geq s^*$,
\begin{align*}
Z' &\geq \left(\frac{1}{2} - \epsilon \right)Z- \epsilon(X + \bar{Y})\\
\left|X'\right| &\leq \epsilon (X + \bar{Y} + Z )\\
\bar{Y}' &\leq - \left(\frac{1}{2} - \epsilon \right)\bar{Y} + \epsilon\left(X + Z\right)
\end{align*}
where $\bar{Y}(s) = Y(s) + r(s)$ with $r(s) = \left\| |y|^{\frac{k}{2}} V^2(s)\right\|_{L^2_\rho}$ for a fixed integer $k$.
\end{lemm}
\begin{proof} From the fact that $|\bar{F}(V,s)| \leq CV^2$ for $s$ large, the proof is the same as the proof of Theorem A, pages 842-847 in Filippas and Kohn \cite{FKcpam92}. \end{proof}

\noindent The following lemma allows us to conclude Proposition \ref{prop:reducetovnull}: 
\begin{lemm}\label{lemm:ODEinFK92} Let $\xi(t), \nu(t), \zeta(t)$ be absolutely continuous, real-valued functions that are nonnegative and satisfy:\\
$i)$ $(\xi(t), \nu(t), \zeta(t)) \to 0$ as $t \to +\infty$,\\
$ii)$ For all $\epsilon > 0$, there exists $t_0 \in \mathbb{R}$ such that for all $t \geq t_0$,
\begin{align*}
\zeta' &\geq c_0\zeta - \epsilon (\xi + \nu)\\
|\xi'| & \leq \epsilon(\xi + \nu + \zeta)\\
\nu' & \leq -c_0\nu + \epsilon(\xi + \zeta),
\end{align*}
for some $c_0 > 0$.\\
Then either $\xi + \zeta = o(\nu)$ or $\nu + \zeta = o(\xi)$ as $t \to +\infty$.
\end{lemm}
\begin{rema} In the first case, we clearly see that $\nu' \leq -\frac{c_0}{2}\nu$ for $t$ large, hence $\xi, \upsilon, \zeta $ tend to zero exponentially fast.
\end{rema}
\begin{proof} The original proof is due to Filippas and Kohn \cite{FKcpam92}. For this particular statement, see Lemma A.1, page 3425 \cite{NZams10} for the proof.
\end{proof}

Since $\|V(s)\|_{L^\infty_{loc}} \to 0$ as $s \to +\infty$, we have $X(s), \bar{Y}(s), Z(s) \to 0 $ as $s \to +\infty$. Thus, Lemma \ref{lemm:ODEinFK92} applies to $X(s), \bar{Y}(s),$ and $Z(s)$ and yields the desired result (use the remark after the statement). This ends the proof of Proposition \ref{prop:reducetovnull}.
\end{proof}
\subsection{Deriving conclusion $ii)$ of Proposition \ref{prop:refinedasymptotic}}
In this part, we recall from Filippas and Liu the proof of $ii)$ of Proposition \ref{prop:refinedasymptotic}. We focus on the case $ii)$ of Proposition \ref{prop:reducetovnull}, namely that
\begin{equation}\label{equ:v_vnull2}
\|V_{+}(s)\|_{L^2_\rho}+ \|V_{-}(s)\|_{L^2_\rho} = o\left(\|V_{null}(s)\|_{L^2_\rho}\right) \quad \text{as} \quad s \to +\infty,
\end{equation}
and show that it leads to case $ii)$ of Proposition \ref{prop:refinedasymptotic}.\\
We first claim the following proposition:
\begin{prop}[\textbf{An ODE satisfied by $V_{null}(s)$ as $s \to +\infty$}] \label{prop:ptvnull} If $\|V_{+}(s)\|_{L^2_\rho}+ \|V_{-}(s)\|_{L^2_\rho} = o\left(\|V_{null}(s)\|_{L^2_\rho}\right)$, then\\
$i)$ for all $i,j \in\{1, ...,n\}$ and as $s \to +\infty$, 
\begin{equation}\label{equ:v_2ijH2ij}
V_{2,ij}'(s) = \frac{p}{2\kappa} \int_{\mathbb{R}}V_{null}^2(y,s) \frac{H_{2,ij}(y)}{\|H_{2,ij}(y) \|_{L^2_\rho}^2}\rho(y) dy + o\left(\|V_{null}(s)\|^2_{L^2_\rho}\right).
\end{equation}
$ii)$ There exist a symmetric $n\times n$ matrix $A(s)$ such that for all $ s \in \mathbb{R}$,
\begin{align}
&V_{null}(y,s) = y^TA(s)y - 2tr(A(s))\nonumber\\
\text{and}\;\; &c_1\|A(s)\| \leq \|V_{null}(s)\|_{L^2_\rho} \leq c_2\|A(s)\|\label{equ:defVnull}
\end{align}
where $c_1, c_2$ are some positive constant and $\|A\|$ stands for any norm on the space of $n \times n$ symmetric matrices. Moreover, 
\begin{equation}\label{equ:odeforA}
A'(s)= \frac{4p}{\kappa}A^2(s)  + o\left(\|A(s)\|^2\right) \quad \text{as} \quad s \to +\infty.
\end{equation}
\end{prop}
\begin{proof} Let us remark that $ii)$ follows directly from $i)$. Here, one has to use \eqref{rem:boundF_H} which is more accurate than \eqref{equ:estFbar2}, in order to isolate the $\mathcal{O}(V^2)$ term in the nonlinear term. Using properties of Hermites polynomials, we may project that term and obtain \eqref{equ:v_2ijH2ij}. 
\end{proof}
In the next step, we show  that although  we can not derive directly from \eqref{equ:v_2ijH2ij} the asymptotic behavior of $V_{null}(s)$, we can use it to show that $\|V(s)\|_{L^2_\rho}$ decays like $\frac{1}{s}$ as $s \to +\infty$. More precisely, we have the following proposition:
\begin{prop}\label{prop:behv} If $\|V_{+}(s)\|_{L^2_\rho}+ \|V_{-}(s)\|_{L^2_\rho} = o\left(\|V_{null}(s)\|_{L^2_\rho}\right)$, then for $s$ large, we have 
\begin{equation}\label{equ:behavCaseA}
\frac{c_1}{s} \leq \|V(s)\|_{L^2_\rho} \leq \frac{c_2}{s},
\end{equation}
for some positive constants $c_1$ and $c_2$.
\end{prop}
\begin{proof} Since $\|V(s)\|_{L^2_\rho} \sim \|V_{null}(s)\|_{L^2_\rho}$ and because of \eqref{equ:defVnull}, it is enough to show that 
\begin{equation}\label{equ:Asdecay}
\frac{c_1}{s} \leq \|A(s)\| \leq \frac{c_2}{s}, \quad \text{for $s$ large}. 
\end{equation} 
Since the proof of \eqref{equ:Asdecay} is totally given in Section 3 of Filippas and Liu \cite{FLaihn93}, we just give its steps of the proof below.
\noindent The following Lemma asserts that $A(s)$ has continuously differential eigenvalues:
\begin{lemm}[\textbf{\cite{RELbook69, KATbook95}}] \label{lemm:Aphi}
Suppose that $A(s)$ is a $n \times n$ symmetric and continuously differentiable matrix-function in some interval $I$, then there exists continuously differentiable functions $\lambda_1(s), \dots , \lambda_n(s)$ in $I$ such that for all $i \in \{1,\dots, n\}$,
$$A(s)\Phi^{(i)}(s) = \lambda_i(s)\Phi^{(i)}(s),$$
for some orthonormal system of vector-functions $\Phi^{(1)}(s), \dots, \Phi^{(n)}(s)$.
\end{lemm}
\noindent Let $\lambda_1(s), \dots , \lambda_n(s)$ be the eigenvalues of $A(s)$. We can derive from \eqref{equ:odeforA} an equation satisfied by $\lambda_i(s),\, i \in \{1, \dots, n\}$:
\begin{lemm}[\textbf{Filippas and Liu \cite{FLaihn93}}]\label{lemm:equinlam}
The eigenvalues of $A(s)$ satisfy for all $i \in \{1, \dots, n\}$,
\begin{equation}\label{equ:wer456l}
\lambda'_i(s) = \frac{4p}{\kappa}\lambda^2_i(s) + o\left( \sum_{i=1}^n \lambda^2_i(s)\right).
\end{equation}
\end{lemm}
\noindent Using \eqref{equ:wer456l}, one can show that (see the end of Section 3 in \cite{FLaihn93})
\begin{equation}\label{equ:behALams}
\frac{c_1}{s} \leq \sum_{i=1}^n |\lambda_i(s)| \leq \frac{c_2}{s}, \quad \text{for $s$ large}.
\end{equation}
Since $\|A(s)\| = \sum_{i=1}^n |\lambda_i(s)|$, this concludes the proof of \eqref{equ:Asdecay} and Proposition \ref{prop:behv} also.
\end{proof}

Using the fact that $\|V(s)\|_{L^2_\rho}$ decays like $\frac{1}{s}$, we will show that $\|V_-(s))\|_{L^2_\rho} + \|V_+(s))\|_{L^2_\rho} $ is in fact $\mathcal{O}(\|V_{null}(s))\|^2_{L^2_\rho}$ and not only $o(\|V_{null}(s))\|_{L^2_\rho})$. This new estimate will be used then to derive a more accurate equation satisfied by $V_{null}$. 
\begin{prop} \label{prop:newEqvnull} If $\|V_{+}(s)\|_{L^2_\rho}+ \|V_{-}(s)\|_{L^2_\rho} = o\left(\|V_{null}(s)\|_{L^2_\rho}\right)$, then we have 
\begin{align}
V_{2,ij}'(s) &= \frac{p}{2\kappa} \int_{\mathbb{R}^n} V^2_{null}(y,s) \frac{H_{2,ij}(y)}{\|H_{2,ij}\|^2_{L^2_\rho} } \rho(y)dy \nonumber\\
& \qquad \qquad + \mathcal{O}\left(\|V_{null}(s)\|^3_{L^2_\rho}\right) +  \mathcal{O}\left(\frac{\|V_{null}(s)\|^2_{L^2_\rho}}{s^{a-1}}\right),\label{equ:newqeq2vnull}
\end{align}
and 
\begin{equation}\label{equ:newqeq2As}
A'(s) = \frac{4p}{\kappa}A^2(s) + \mathcal{O}\left(\frac{1}{s^3}\right) + \mathcal{O}\left(\frac{1}{s^{a+1}}\right),
\end{equation}
where $A(s)$ is given in \eqref{equ:defVnull}.
\end{prop}
\begin{proof} The proof corresponds to Section 4 in \cite{FLaihn93}. Let us mention that the proof relies on the following priori estimate of solutions of \eqref{equ:V} shown by Herrero and Vel\'azquez in \cite{HVaihn93}. Although they proved their result in the case $N = 1$, their proof holds in higher dimensions under the general assumption \eqref{equ:estFbar2}.
\begin{lemm}[\textbf{Herrero and Val\'azquez \cite{HVaihn93}}]\label{lemm:HVslow} Assume that $V$ solves \eqref{equ:V} and $|V| \leq M < +\infty$. Then for any $r >1$, $q > 1$ and $L >0$, there exist $s_0^* = s_0^*(q,r)$ and $C = C(q,r,L) >0$ such that 
\begin{equation*}
\left(\int_{\mathbb{R}^n} |V(y, s + \tau)|^r \rho(y) dy\right)^{\frac{1}{r}} \leq C\left(\int_{\mathbb{R}^n} |V(y, s)|^q \rho(y) dy\right)^{\frac{1}{q}},
\end{equation*}
for any $s\geq 0$ and any $\tau \in [s_0^*, s_0^* + L]$.
\end{lemm}
\noindent From Proposition \ref{prop:behv}, we have $\|V(s)\|_{L^2_\rho}$ decays like $\frac{1}{s}$. Then Lemma \ref{lemm:HVslow} implies that 
\begin{equation}\label{equ:slowdecayV}
\left(\int_{\mathbb{R}^n} |V(y, s)|^r \rho(y) dy\right)^{\frac{1}{r}} \leq C\left(\int_{\mathbb{R}^n} |V(y, s)|^q \rho(y) dy\right)^{\frac{1}{q}},
\end{equation}
for any $r > 1$, $q > 1$ and for $s$ large.\\
Using estimate \eqref{equ:slowdecayV}, we derive the fact that 
\begin{equation}\label{equ:V2new}
\|V_{+}(s)\|_{L^2_\rho}+ \|V_{-}(s)\|_{L^2_\rho} = \mathcal{O}\left(\|V_{null}(s)\|^2_{L^2_\rho}\right).
\end{equation}
Then, projecting \eqref{equ:V} onto the null space of $\mathcal{L}$ and using \eqref{equ:V2new}, \eqref{equ:slowdecayV}, we would obtain \eqref{equ:newqeq2vnull}. Since $\|V(s)\|_{L^2_\rho} \sim \|V_{null}(s)\|_{L^2_\rho} \sim \frac{1}{s}$, we then obtain \eqref{equ:newqeq2As} from \eqref{equ:newqeq2vnull}. This ends the proof of Proposition \ref{prop:newEqvnull}.
\end{proof}
Let us now use \ref{prop:newEqvnull} to derive conclusion $ii)$ of Proposition \ref{prop:refinedasymptotic}. Using Lemma \ref{lemm:Aphi}, we get from \eqref{equ:newqeq2As} that the eigenvalues of $A(s)$ satisfy
$$
\forall i \in \{1, \dots, n\}, \;\; \lambda'_i(s) = \frac{4p}{\kappa}\lambda_i^2(s) + \mathcal{O}\left(\frac{1}{s^{a+1}}\right)+ \mathcal{O}\left(\frac{1}{s^3}\right), \quad \text{as} \;\; s \to +\infty,
$$
then Lemma \ref{lemm:elementary} yields
\begin{align}
\text{either} &\;\; \lambda_i(s) = -\frac{\kappa}{4ps} +  \mathcal{O}\left(\frac{1}{s^{a}}\right)\;\;\text{or} \;\; \lambda_i(s) = \mathcal{O}\left(\frac{1}{s^{a}}\right),\;\;\text{if} \; a \in (1,2), \label{eq:tmpLa12}\\
\text{either} &\;\; \lambda_i(s) = -\frac{\kappa}{4ps} + \mathcal{O}\left(\frac{\log s}{s^2}\right) \;\; \text{or} \;\; \lambda_i(s) = \mathcal{O}\left(\frac{1}{s^2}\right), \;\; \text{if} \; a \geq 2.\label{eq:tmpLap2}
\end{align}
Therefore, Proposition 5.1 in \cite{FLaihn93} yields the existence of $l \in \{1, \dots, n\}$ and a $n \times n$ orthonormal matrix $Q$ such that 
\begin{align*}
A(s) &= -\frac{\kappa}{4ps}A_l + \mathcal{O}\left(\frac{1}{s^{a}}\right),\;\;\text{if} \; a \in (1,2),\\
A(s) &= -\frac{\kappa}{4ps}A_l + \mathcal{O}\left(\frac{\log s}{s^2}\right), \;\; \text{if} \; a \geq 2,
\end{align*}
where $$A_l = Q\left(\begin{array}{cc}
\mathbf{I}_{l} & O\\
O & O
\end{array} \right) Q^{-1}.$$
Combining this with \eqref{equ:defVnull}, it yields the behavior of $V_{null}(y,s)$ and $V(y,s)$ announced in $ii)$ of Proposition \ref{prop:refinedasymptotic}. The convergence in $\mathcal{C}^{k,\gamma}_{loc}$ follows from standard parabolic regularity (see section 5 in \cite{FLaihn93} for a brief demonstration). This completes the proof of $ii)$ of Proposition \ref{prop:refinedasymptotic}.

\subsection{Deriving conclusions $i)$ and $iii)$ of Proposition \ref{prop:refinedasymptotic}}
In this part, we recall the proof given by Vel\'azquez \cite{VELtams93}. We focus on the case $i)$ of Proposition \ref{prop:reducetovnull}, namely $\|V(s)\|_{L^2_\rho} = \mathcal{O}(e^{-\mu s})$ for some $\mu > 0$, and we will show that it leads to either $i)$ or $iii)$ of Proposition \ref{prop:refinedasymptotic}. Let us start the first step. From equation \eqref{equ:V}, we write $V(y,s)$ in the integration form
$$V(y,s) = S_\mathcal{L}(s)V(s_0) + \int_{s_0}^sS_\mathcal{L}(s -\tau)\bar{F}(V(\tau), \tau)d\tau, \quad \text{with} \quad s_0 = -\log T,$$
where $S_\mathcal{L}(s)$ is the linear semigroup corresponding to the heat-type equation $\partial V = \mathcal{L}V$ given by
$$S_\mathcal{L}(s)V(y,\tau) = \sum_{|\alpha| = 0}^\infty a_\alpha(\tau)e^{\left(1 - \frac{|\alpha|}{2} \right)(s -\tau)}H_\alpha(y),$$
with  
$$a_\alpha(\tau) = \langle V(\tau),H_\alpha\rangle := \int_{\mathbb{R}^n}V(y,\tau)H_\alpha(y)\rho(y)dy.$$

\noindent Let us fix a integer $k_0 > 2$ such that $\frac{k_0}{2} - 1 < 2\mu < \frac{k_0 + 1}{2} - 1$ and write $V(y,s)$ as follow:
\begin{align*}
V(y,s) &= \sum_{|\alpha| \leq k_0} a_\alpha(s_0)e^{\left(1 - \frac{|\alpha|}{2} \right)(s -s_0)}H_\alpha(y) + \sum_{|\alpha| \geq k_0 + 1} a_\alpha(s_0)e^{\left(1 - \frac{|\alpha|}{2} \right)(s -s_0)}H_\alpha(y)\\
& + \sum_{|\alpha| \leq k_0}H_\alpha(y) \int_{s_0}^s e^{\left(1 - \frac{|\alpha|}{2} \right)(s - \tau)}\langle \bar{F}(V(y,\tau), \tau), H_\alpha(y)\rangle d\tau \\
& +  \sum_{|\alpha| \geq k_0 + 1}H_\alpha(y) \int_{s_0}^s e^{\left(1 - \frac{|\alpha|}{2} \right)(s - \tau)}\langle \bar{F}(V(y,\tau), \tau), H_\alpha(y)\rangle d\tau\\
& := I + II + III + IV.
\end{align*}
Since $|\bar{F}(V,s)| \leq C|V|^2$ and $\|V(s)\|_{L^2_\rho} \leq Ce^{-\mu s}$, we derive from Lemma \ref{lemm:HVslow} that 
\begin{equation}\label{equ:esttmpFb}
\|\bar{F}(V(\cdot,\tau))\|_{L^2_\rho} \leq Ce^{-2\mu \tau}.
\end{equation}
By a direct computation, we find that 
$$\|II\|_{L^2_\rho} + \|IV\|_{L^2_\rho} \leq Ce^{-2\mu s}, \quad \text{for some $C > 0$}.$$
For $III$, we write 
\begin{align*}
&\int_{s_0}^s e^{-\left(1 - \frac{|\alpha|}{2} \right)\tau}\langle \bar{F}(V(y,\tau), \tau), H_\alpha(y)\rangle d\tau\\
& \qquad \qquad = \beta_\alpha - \int_{s}^{+\infty} e^{-\left(1 - \frac{|\alpha|}{2} \right)\tau}\langle \bar{F}(V(y,\tau), \tau), H_\alpha(y)\rangle d\tau.
\end{align*}
Using \eqref{equ:esttmpFb}, we can bound the last term of the above expression by $Ce^{-2\mu s}$. Hence, 
$$V(y,s) = \sum_{|\alpha| \leq k_0} (a_\alpha + \beta_\alpha)e^{\left(1 - \frac{|\alpha|}{2}\right)s}H_\alpha(y) + Q(y,s),$$
where $\|Q(y,s)\|_{L^2_\rho} = \mathcal{O}(e^{-2 \mu s})$.\\
Since $\|V(y,s)\|_{L^2_\rho} = \mathcal{O}(e^{-\mu s})$, it requires $a_\alpha + \beta_\alpha = 0$ for $|\alpha| \leq 2$. Thus, we have two possibilities: if there exists an integer $m \in [3, k_0]$ such that $a_\alpha + \beta_\alpha \ne 0$ for $|\alpha| = m$ and $a_\alpha + \beta_\alpha = 0$ for all $|\alpha| < m$, then we obtain $iii)$ of Proposition \ref{prop:refinedasymptotic} for some $m \in [3, k_0]$. If this is not the case, we get $\|V(y,s)\|_{L^2_\rho} = \mathcal{O}(e^{-2 \mu s})$. Using this new estimate and repeating the process in a finite number of steps, we may obtain either $iii)$ of Proposition \ref{prop:refinedasymptotic} for some $m \geq 3$ or $\|V(y,s)\|_{L^2_\rho} = \mathcal{O}(e^{-Rs})$ for any $R > 0$. For the second case, we use the  following nondegeneracy result from Herrero and Vel\'azquez  \cite{HVaihn93} in order to conclude that $V(y,s) \equiv 0$, which is $i)$ of Proposition \ref{prop:refinedasymptotic},
\begin{lemm}[\textbf{Herrero and Vel\'azquez  \cite{HVaihn93}}] Let $V$ be a solution to equation \eqref{equ:V}. Assume that $|V(y,s)|$ is bounded, and that for any $R > 0$ there exists $C = C(R)$ such that 
$$\|V(s)\|_{L^2_\rho} \leq Ce^{-Rs} \quad \text{if} \quad s \geq 0,$$
then $V(y,s)\equiv 0$.
\end{lemm}
\begin{proof} Since the proof written in \cite{HVaihn93} holds under general assumption \eqref{equ:estFbar2}, we then refer the reader to Lemma 3.5, page 144 of \cite{HVaihn93} for detail of the proof.
\end{proof}
Since the convergence in $\mathcal{C}^{k,\gamma}_{loc}$ for any $k \geq 1$ and $\gamma \in (0,1)$ follows from a standard parabolic regularity, we end the proof of Proposition \ref{prop:refinedasymptotic} here. This also concludes the proof of  Theorem \ref{theo:refinedasymptotic}.
 
\section{Bow-up profile for equation \eqref{equ:problem} in extended spaces regions}\label{sec:ext}
We give the proof of Theorem \ref{theo:pro} in this section. Note that the derivation of Theorem \ref{theo:pro} from Theorem \ref{theo:refinedasymptotic} in the unperturbed case ($h \equiv 0$) was done by Vel\'azquez in \cite{VELcpde92}. The idea to extend the  convergence up to sets of the type $\{|y| \leq K_0 \sqrt{s}\}$ or $\{|y| \leq K_0e^{\left(\frac{1}{2} - \frac{1}{m}\right)s}\}$ is to estimate the effect of the convective term $-\frac{y}{2}\cdot\nabla w$ in the equation \eqref{equ:divw1} in $L^q_\rho$ spaces with $q > 1$. Since the proof of Theorem \ref{theo:pro} is actually in spirit by the method given in \cite{VELcpde92}, all that we need to do is to control the strong perturbation term in equation \eqref{equ:divw1}. We therefore give the main steps of the proof and focus only on the new arguments. The proof will be separated into two parts: the first part concerns case $ii)$ in Theorem \ref{theo:refinedasymptotic} and gives the asymptotic behavior of $w$ in the $\frac{y}{\sqrt{s}}$ variable, and the second part concerns case $iii)$ in Theorem \ref{theo:refinedasymptotic} and gives the asymptotic behavior of $w$ in the $ye^{-\left(\frac{1}{2} - \frac{1}{m} \right)s}$ variable. In Part 1, we stick to the method of Vel\'azquez \cite{VELcpde92}, whereas, in Part 2, where we work in the scale $e^{-\mu s}$ for $\mu > 0$, we need new ideas to get rid of the term in the scale $\frac 1 s$ coming from the strong perturbation.\\

\noindent \textbf{Part 1: Case $ii)$ in Theorem \ref{theo:refinedasymptotic} and asymptotic behavior in the $\frac{y}{\sqrt{s}}$ variable}.\\
Let us restate $i)$ of Theorem \ref{theo:pro} in the following proposition:
\begin{prop}[\textbf{Asymptotic behavior in the $\frac{y}{\sqrt{s}}$ variable}] \label{prop:1} Assume that $w$ is a solution of equation \eqref{equ:divw1} which satisfies $ii)$ of Theorem \ref{theo:refinedasymptotic}. Then, for all $K > 0$, 
$$\sup_{|\xi| \leq K} \left|w(\xi\sqrt{s},s) - f_l(\xi)\right| = \mathcal{O}\left(\frac{1}{s^{a-1}} \right) + \mathcal{O}\left(\frac{\log s}{s} \right), \quad \text{as} \quad s \to + \infty,$$
where $f_l(\xi) = \kappa\left(1 + \frac{p-1}{4p}\sum_{j = 1}^l\xi_j^2\right)^{-\frac{1}{p-1}}$. 
\end{prop}
\begin{proof} Following the method in \cite{VELcpde92}, we define $q = w - \varphi$, where
\begin{equation}\label{equ:defiv}
\varphi(y,s) = \kappa\left(1+ \frac{p-1}{4ps}\sum_{j=1}^l y_j^2\right)^{-\frac{1}{p-1}} + \frac{\kappa l}{2ps}.
\end{equation}
Using Taylor's formula in \eqref{equ:defiv} and $ii)$ of Theorem \ref{theo:refinedasymptotic}, we find that 
\begin{equation}
\|q(y,s)\|_{L^2_\rho} = \mathcal{O}\left(\frac{1}{s^{a}} \right) + \mathcal{O}\left(\frac{\log s}{s^2} \right), \quad \text{as} \quad s \to + \infty.
\end{equation}
Straightforward calculations based on equation \eqref{equ:divw1} yield 
\begin{equation}\label{eq:Wa}
\partial_s q = (\mathcal{L} + \omega)q + F(q) + G(q,s) + R(y,s), \quad \forall (y,s) \in \mathbb{R}^n \times [-\log T, +\infty), 
\end{equation}
where 
\begin{align*}
\omega(y,s) & = p(\varphi^{p-1} - \kappa^{p-1}) + e^{-s}h'\left(e^\frac{s}{p-1}\varphi\right),\\
F(q)& = |q + \varphi|^{p-1}(q + \varphi) - \varphi^p -p\varphi^{p-1}q,\\
G(q,s)& = e^{-\frac{ps}{p-1}}\left[h\left(e^\frac{s}{p-1}(q + \varphi)\right) - h\left(e^\frac{s}{p-1}\varphi\right) - e^\frac{s}{p-1}h'\left(e^\frac{s}{p-1}\varphi\right)q\right],\\
R(y,s)&= -\partial_s \varphi + \Delta \varphi - \frac{y}{2}\cdot \nabla \varphi - \frac{\varphi}{p-1} + \varphi^p + e^{-\frac{ps}{p-1}}h\left(e^\frac{s}{p-1}\varphi\right).
\end{align*}
Let $K_0 > 0$ be fixed, we consider first the case $|y| \geq 2K_0\sqrt{s}$ and then $|y| \leq 2K_0\sqrt{s}$ and make a Taylor expansion for $\xi = \frac{y}{\sqrt{s}}$ bounded. Simultaneously, noticing from \eqref{equ:estimateH}, we then obtain for all $s \geq s_0$, 
$$\omega(y,s) \leq \frac{C_1}{s}, $$
$$|F(q)| + |G(q,s)| \leq C_1(q^2+ \mathbf{1}_{\{|y| \geq 2K_0\sqrt{s}\}}),$$
$$|R(y,s)| \leq C_1\left( \frac{|y|^2}{s^2} + \frac{1}{s^2} + \frac{1}{s^a} + \mathbf{1}_{\{|y| \geq 2K_0\sqrt{s}\}}\right),$$
where $C_1 = C_1(M_0, K_0) > 0$, $M_0$ is the bound of $w$ in $L^\infty$-norm.\\
Let $Q = |q|$, we then use the above estimates and Kato's inequality, i.e  $\Delta f \cdot \text{sign}(f) \leq \Delta(|f|)$, to derive from equation \eqref{eq:Wa} the following: for all $K_0 > 0$ fixed, there are $C_* = C_*(K_0,M_0) > 0$ and a time $s' > 0$ large enough such that for all $s \geq s_* = \max\{s',-\log T\}$,
\begin{equation}\label{eq:Qa}
\partial_s Q \leq \left(\mathcal{L} + \frac{C_*}{s} \right)Q + C_*\left(Q^2 + \frac{|y|^2}{s^2} + \frac{1}{s^2}+ \frac{1}{s^a} + \mathbf{1}_{\{|y| \geq 2K_0\sqrt{s}\}} \right), \quad \forall y \in \mathbb{R}^n.
\end{equation}
Since 
$$\left|w(y,s) - f_l\left(\frac{y}{\sqrt{s}}\right) \right| \leq Q + \frac{\kappa l}{2ps},$$
the conclusion of  Proposition \ref{prop:1} follows if we show that 
\begin{equation}\label{equ:pr1}
\forall K_0 > 0,\quad \sup_{|y| \leq K_0\sqrt{s}} Q(y,s) \to 0 \quad \text{as} \quad s \to + \infty.
\end{equation}
Let us now focus on the proof of \eqref{equ:pr1} in order to conclude Proposition \ref{prop:1}. For this purpose, we introduce the following norm: for $r \geq 0$, $q > 1$ and $f \in L^q_{loc}(\mathbb{R}^n)$, 
$$L_\rho^{q,r}(f) \equiv \sup_{|\xi| \leq r }\left(\int_{\mathbb{R}^n}|f(y)|^q\rho(y - \xi)dy \right)^\frac{1}{q}.$$
Following the idea in \cite{VELcpde92}, we shall make estimates on solution of \eqref{eq:Qa} in the $L^{2,r(\tau)}_\rho$ norm where $r(\tau) = K_0e^{\frac{\tau - \bar{s}}{2}} \leq K_0\sqrt{\tau}$. Particularly, we have the following:
\begin{lemm}\label{lemm:gt} Let $s$ be large enough and $\bar{s}$ is defined by $e^{s - \bar{s}} = s$. Then for all $\tau \in [\bar{s}, s]$ and for all $K_0 > 0$, it holds that 
$$ g(\tau) \leq 
C_0\left( e^{\tau - \bar{s}}\epsilon(\bar{s}) + \int_{\bar{s}}^{(\tau - 2K_0)_+}\frac{e^{(\tau - t - 2K_0)}g^2(t)}{\left(1 - e^{-(\tau - t - 2K_0)} \right)^{1/20}}dt \right)$$
where $g(\tau) = L^{2,r(K_0,\tau,\bar{s})}_\rho (Q(\tau))$, $r(K_0, \tau, \bar{s}) = K_0e^{\frac{\tau - \bar{s}}{2}}$, $\epsilon(s) = \mathcal{O}\left(\frac{1}{s^{a}} \right) + \mathcal{O}\left(\frac{\log s}{s^2} \right)$, $C_0 = C_0(C_*, M_0, K_0)$ and $z_+ = \max\{z,0\}$.
\end{lemm}
\begin{proof} Multiplying \eqref{eq:Qa} by $\alpha(\tau) = e^{\int_{\bar{s}}^\tau \frac{C_*}{t}}dt$, then we write $Q(y,\tau)$ for all $(y,\tau) \in \mathbb{R}^n \times [\bar{s},s]$ in the integration form:
\begin{align*}
Q(y,\tau) &= \alpha(\tau)S_\mathcal{L}(\tau - \bar{s})Q(y,\bar{s})\\
&+ C_*\int_{\bar{s}}^\tau \alpha(\tau)S_\mathcal{L}(\tau - t)\left(Q^2 + \frac{|y|^2}{t^2} + \frac{1}{t^2}+ \frac{1}{t^a} + \mathbf{1}_{\{|y| \geq 2K_0\sqrt{t}\}} \right)dt,
\end{align*}
where $S_\mathcal{L}$ is the linear semigroup corresponding to the operator $\mathcal{L}$.\\
Next, we take the $L^{2, r(K_0, \tau, \bar{s})}_\rho$-norms both sides in order to get the following:
\begin{align*}
g(\tau) &\leq C_0L^{2,r}_\rho \big[S_\mathcal{L}(\tau -\bar{s})Q(\bar{s})\big] + C_0\int_{\bar{s}}^\tau L^{2,r}_\rho \big[S_\mathcal{L}(\tau -t)Q^2(t)\big]dt \\
&+ C_0\int_{\bar{s}}^\tau L^{2,r}_\rho \left[S_\mathcal{L}(\tau -t)\left(\frac{|y|^2}{t^2} + \frac{1}{t^2}+ \frac{1}{t^a}\right) \right]dt\\
&+ C_0\int_{\bar{s}}^\tau L^{2,r}_\rho \big[S_\mathcal{L}(\tau -t)\mathbf{1}_{\{|y| \geq 2K_0\sqrt{t}\}}\big]dt\\
& \equiv J_1 + J_2 + J_3 + J_4.
\end{align*}
Proposition 2.3 in \cite{VELcpde92} (with a slight modification for the estimate of $J_3$) yields
$$|J_1| \leq C_0e^{\tau - \bar{s}}\|Q(\bar{s})\|_{L^2_\rho} = e^{\tau - \bar{s}}\mathcal{O}(\epsilon(\bar{s})) \quad \text{as} \quad \bar{s} \to + \infty,$$
$$|J_2| \leq \frac{C_0}{\bar{s}^2}e^{\tau - \bar{s}} + C_0\int_{\bar{s}}^{(\tau - 2K_0)_+}\frac{e^{(\tau - t - 2K_0)}}{\left(1 - e^{-(\tau - t - 2K_0)} \right)^{1/20}}\left[L_\rho^{2,r(K_0, t, \bar{s})}Q(t)\right]^2dt,$$
$$|J_3| \leq C_0e^{\tau - \bar{s}}\left(\frac{1}{\bar{s}^2} + \frac{1}{\bar{s}^a}\right)(1 + (\tau - \bar{s})),$$
$$|J_4| \leq C_0e^{-\delta \bar{s}}, \quad \text{where} \quad \delta = \delta(K_0) > 0.$$
Putting together the estimates on $J_i, i =1, 2, 3, 4$, we conclude the proof of Lemma \ref{lemm:gt}.
\end{proof}
\noindent We now use the following Gronwall lemma from Vel\'azquez \cite{VELcpde92}:
\begin{lemm}[\textbf{Vel\'azquez \cite{VELcpde92}}]\label{lem:Gro} Let $\epsilon, C, R$ and $\delta$ be positive constants, $\delta \in (0,1)$. Assume that $H(\tau)$ is a family  of continuous functions satisfying
$$\mathcal{H}(\tau) \leq \epsilon e^\tau + C\int_0^{(\tau - R)_+} \frac{e^{\tau -s} \mathcal{H}^2(s)}{\left(1 - e^{-(\tau - s - R)} \right)^\delta}ds, \quad \text{for $\tau > 0$}.$$
Then there exist $\theta = \theta(\delta, C, R)$ and $\epsilon_0 = \epsilon_0(\delta, C, R)$ such that for all $\epsilon \in (0,\epsilon_0)$ and any $\tau$ for which $\epsilon e^\tau \leq \theta$, we have
$$\mathcal{H}(\tau) \leq 2\epsilon e^\tau.$$
\end{lemm}

\noindent Applying Lemma \ref{lem:Gro} with $\mathcal{H} \equiv g$, we see from Lemma \ref{lemm:gt} that for $s$ large enough,
$$g(\tau) \leq 2C_0e^{\tau - \bar{s}}\epsilon(\bar{s}), \quad \forall \tau \in [\bar{s},s].$$
If $\tau = s$, then $e^{s - \bar{s}} = s$, $r = K_0\sqrt{s}$ and 
$$g(s) \equiv L^{2, K_0\sqrt{s}}_\rho\big(Q(s)\big) = \mathcal{O}\left(\frac{1}{s^{a-1}} \right) + \mathcal{O}\left(\frac{\log s}{s} \right),\; \text{as} \quad s \to + \infty.$$
By using the regularizing effects of the semigroup $S_{\mathcal{L}}$ (see Proposition 2.3 in \cite{VELcpde92}), we then obtain
$$\sup_{|y| \leq \frac{K_0\sqrt{s}}{2}} Q(y,s) \leq C'(C_*,K_0, M_0)L^{2,K_0\sqrt{s}}_\rho(Q(s)) =  \mathcal{O}\left(\frac{1}{s^{a-1}} \right) + \mathcal{O}\left(\frac{\log s}{s} \right),$$
as $s \to +\infty$, which concludes the proof of Proposition \ref{prop:1}.
\end{proof}

\noindent \textbf{Part 2: Case $iii)$ in Theorem \ref{theo:refinedasymptotic} and the asymptotic behavior in the $ye^{-\left(\frac{1}{2} - \frac{1}{m} \right)s}$ variable}.\\ 
We give the proof of $ii)$ of Theorem \ref{theo:pro} in this part. Since we work in the scale $e^{-\mu s}$ for $\mu > 0$ in the case where $iii)$ in Theorem \ref{theo:refinedasymptotic} occurs, we need new ideas to get rid of the term in the scale $\frac 1 s$ coming from the strong perturbation.\\ 

Let us restate $ii)$ of Theorem \ref{theo:pro} in the following proposition:
\begin{prop}[\textbf{Asymptotic behavior in the $ye^{-\left(\frac{1}{2} - \frac{1}{m} \right)s}$ variable}] \label{prop:2} Assume that $w$ is a solution of equation \eqref{equ:divw1} and satisfies $iii)$ of Theorem \ref{theo:refinedasymptotic}. Then, for all $K > 0$, 
\begin{equation}\label{equ:asypsim}
\sup_{|\xi| \leq K} \left|w(\xi e^{\left(\frac{1}{2} - \frac{1}{m}\right)s},s) - \psi_m(\xi)\right| \to 0, \quad \text{as} \quad s \to + \infty,
\end{equation}
where $\psi_m(\xi)= \kappa\left(1 + \kappa^{-p}\sum \limits_{|\alpha| = m}c_\alpha \xi^\alpha \right)^{-\frac{1}{p-1}}$, and $m \geq 4$ is an even integer.
\end{prop}
\begin{proof} Note that the proof will proceed in the same way as for the proof of Proposition \ref{prop:1}. Let us introduce $q = w - \varphi$, where
\begin{equation}\label{equ:defivarphi}
\varphi(y,s) = \frac{\phi(s)}{\kappa}J(y,s),
\end{equation}
with
$$J(y,s) = \frac{\phi(s)}{\kappa}\left[G\left(ye^{-\left(\frac{1}{2} - \frac{1}{m}\right)s}\right) + e^{-\left(\frac{m}{2} - 1\right)s} \left( \sum_{|\alpha| = m}c_\alpha y^\alpha - \sum_{|\alpha| = m}c_{\alpha}H_\alpha(y)\right)\right],$$
and $G(\xi) =  \kappa\left(1 + \kappa^{-p}\sum \limits_{|\alpha| = m}c_\alpha \xi^\alpha \right)^{-\frac{1}{p-1}}$ satisfying 
\begin{equation}\label{eq:Gxi}
-\frac{\xi}{m}\cdot \nabla G(\xi) + G^p(\xi) = \frac{G(\xi)}{p-1}.
\end{equation}
Note that Vel\'azquez \cite{VELcpde92} takes $\varphi = J$, and if we do the same, we will obtain some terms in the scale of $\frac 1 s$, much stronger than the $e^{-\mu s}$ scale that we intended to work in.\\

\noindent Using Taylor's formula in \eqref{equ:defivarphi} and recalling from Lemma \ref{ap:lemmA2} that fact that $\frac{\phi(s)}{\kappa} = 1 + \mathcal{O}(s^{-a})$ as $s \to + \infty$, we have by $iii)$ of Theorem \ref{theo:refinedasymptotic},
\begin{equation}
\|q                                                                                                                                                                                                                                                                                                                                                                                                                                                                                                                                                                                                                                                                                                                                                                                                                                                                                                                                                                                                                                                                                                                                                                                                                                                                                                                                                                                                                                                                                                                                                                                                                                                                                                                                                                                                                                                                                                                                                                                                                                                                                                                                                                                                                                                                                                                                                                                                                                                                                                                                                                                                                                                                                                                                                                                                                                                                                                                                                                                                                                                                                                                                                                                                                                                                                                                                                    (y,s)\|_{L^2_\rho} = o\left(e^{-\left(\frac{m}{2} - 1\right)s}\right), \quad \text{as $s \to +\infty$}.
\end{equation}
Straightforward calculations based on equation \eqref{equ:divw1} yield
\begin{equation}\label{eq:Wa2}
\partial_s q = (\mathcal{L} + \omega)q + F(q) + G(q,s) + R(y,s), \quad \forall (y,s) \in \mathbb{R}^n \times [-\log T, +\infty), 
\end{equation}
where 
\begin{align*}
\omega(y,s) & = p(\varphi^{p-1} - \kappa^{p-1}) + e^{-s}h'\left(e^\frac{s}{p-1}\varphi\right),\\
F(q)& = |q + \varphi|^{p-1}(q + \varphi) - \varphi^p -p\varphi^{p-1}q,\\
G(q,s)& = e^{-\frac{ps}{p-1}}\left[h\left(e^\frac{s}{p-1}(q + \varphi)\right) - h\left(e^\frac{s}{p-1}\varphi\right) - e^\frac{s}{p-1}h'\left(e^\frac{s}{p-1}\varphi\right)q\right],\\
R(y,s)&= -\partial_s \varphi + \Delta \varphi - \frac{y}{2}\cdot \nabla \varphi - \frac{\varphi}{p-1} + \varphi^p + e^{-\frac{ps}{p-1}}h\left(e^\frac{s}{p-1}\varphi\right).
\end{align*}
Fix now $K_0 > 0$ and define $\chi(y,s) = 1$ if $|y| \geq 2K_0e^{\left(\frac{1}{2} - \frac{1}{m} \right)s}$ and $\chi(y,s) = 0$ otherwise. Then, using Taylor's formula for $\xi = ye^{-\left(\frac{1}{2} - \frac{1}{m} \right)s}$ bounded, and noticing from \eqref{equ:estimateH}, we then obtain for all $s \geq s_0$, 
$$\omega(y,s) \leq \frac{C_1}{s},$$
$$|F(q)| + |G(q,s)| \leq C_1\left(q^2 + \chi(y,s)\right),$$
where $C_1 = C_1(M_0,K_0) > 0$.\\
To estimate $R(y,s)$, we write $R(y,s)$ as follow: 
\begin{align*}
R(y,s) & = \frac{\phi(s)}{\kappa}\left(-\partial_s J + \Delta J - \frac{y}{2}\cdot \nabla J - \frac{J}{p-1} + J^p\right)\\
& + \left(- \frac{\phi'(s)}{\kappa}J - \frac{\phi(s)}{\kappa}J^p + \varphi^p + e^{-\frac{ps}{p-1}}h\left(e^\frac{s}{p-1}\varphi\right)\right)\equiv \frac{\phi(s)}{\kappa} I + II.
\end{align*}
By using Taylor's formula, \eqref{eq:Gxi} and Hermite's equation, i.e. 
$$\mathcal{L}H_\alpha(y) = \left(1 - \frac{|\alpha|}{2}\right)H_\alpha(y),$$
it was proved in \cite{VELcpde92} (see Proposition 2.4) that
$$I \leq C e^{-(m-2)s} (|y|^{2m-2} + 1)(1 - \chi(y,s)) + C\chi(y,s), \quad \text{for some $C > 0$}.$$
It remains to estimate $II$. To do so, we write $J(y,s)$ for $|y|e^{\left(\frac{1}{2} - \frac{1}{m}\right)s}$ bounded in the form:
$$J(y,s) = \kappa - e^{-\left(\frac{m}{2}-1\right)s}\sum_{|\alpha| = m}c_{\alpha}H_\alpha(y) + \mathcal{O}\left(e^{-(m - 2)s} |y|^{2m} \right).$$
We then use Taylor's formula in $II$, \eqref{equ:estimateH}, and the fact that $\phi(s)$ satisfies \eqref{equ:phiODE} to find that 
$$II \leq \frac{C}{s^a}e^{-\left(\frac{m}{2}-1 \right)s}(|y|^m + 1)(1 - \chi(y,s)) + C\chi(y,s),  \quad \text{for some $C > 0$}.$$
Note that $e^{-(m-2)s}|y|^{2m-2}(1 - \chi(y,s)) \leq \frac{1}{s^a}e^{-\left(\frac{m}{2}-1 \right)s}|y|^m (1 - \chi(y,s))$ for $s$ large, we then obtain 
$$|R(y,s)| \leq C\left(\frac{1}{s^a}e^{-\left(\frac{m}{2}-1 \right)s}(|y|^m + 1)(1 - \chi(y,s)) + \chi(y,s)\right), \;\; \text{for some $C > 0$}.$$
Let $Q = |q|$ and use Kato's inequality, we obtain from \eqref{eq:Wa2} and from the above estimates that: for all $K_0 > 0$ fixed, there are $C_* = C_*(K_0, M_0) > 0$ and a time $s' > 0$ large enough such that for all $s \geq s_* = \max\{s',-\log T\}$, 
\begin{equation}
\partial_s Q \leq \left(\mathcal{L} + \frac{C_*}{s}\right)Q + C_*\left(Q^2 + \frac{1}{s^a}e^{-\left(\frac{m}{2}-1 \right)s}(|y|^m + 1) + \chi(y,s)\right), \;\; \forall y \in \mathbb{R}^n.
\end{equation}
We claim the following:
\begin{lemm}\label{lemm:gt12} Let $s$ be large enough and $\bar{s} = \frac{2s}{m}$. Then for all $\tau \in [\bar{s}, s]$, $\tau - \bar{s} \geq 2$ and for all $K_0 > 0$, it holds that 
$$ g(\tau) \leq 
e^{\tau - \bar{s}}\left(o\left(e^{-\left(\frac{m}{2} - 1\right)\bar{s}}\right) + C'\int_{\bar{s}}^{(\tau - 2K_0)_+}\frac{e^{(\tau - t - 2K_0)}g^2(t)}{\left(1 - e^{-(\tau - t - 2K_0)} \right)^{1/20}}dt \right)$$
where $g(\tau) = L^{2,r(K_0,\tau,\bar{s})}_\rho (Q(\tau))$, $r(K_0, \tau, \bar{s}) = K_0e^{\frac{\tau - \bar{s}}{2}}$, $C' = C'(C_*, M_0, K_0)$ and $z_+ = \max\{z,0\}$.
\end{lemm}
\begin{proof} Proceeding as in the proof of Lemma \eqref{lemm:gt}, we write 
\begin{align*}
L^{2,r}_\rho(Q) &\leq C_0L^{2,r}_\rho \big[S_\mathcal{L}(\tau -\bar{s})Q(\bar{s})\big] + C_0\int_{\bar{s}}^\tau L^{2,r}_\rho \big[S_\mathcal{L}(\tau -t)Q^2(t)\big]dt \\
&+ C_0\int_{\bar{s}}^\tau L^{2,r}_\rho \left[S_\mathcal{L}(\tau -t)\left(\frac{1}{t^a}e^{-\left(\frac{m}{2} - 1 \right)t}(|y|^m + 1) \right) \right]dt\\
&+ C_0\int_{\bar{s}}^\tau L^{2,r}_\rho \big[S_\mathcal{L}(\tau -t)\chi(y,t)\big]dt\\
& \equiv J_1 + J_2 + J_3 + J_4.
\end{align*}
One can show that for $\bar{s}$ large enough (see Proposition 2.4 in \cite{VELcpde92}),
$$|J_1| = e^{\tau - \bar{s}}o\left(e^{(1 - m/2)\bar{s}}\right),$$
$$|J_2| \leq Ce^{2\tau - 2(m-1)\bar{s}} + C \int_{\bar{s}}^{(\tau - 2K_0)_+}\frac{e^{(\tau - t - 2K_0)}g^2(t)}{\left(1 - e^{-(\tau - t - 2K_0)} \right)^{1/20}}dt,$$
$$|J_3| \leq Ce^{\tau - \bar{s}}\frac{e^{(1 - m/2)\bar{s}}}{\bar{s}^a} \left( 1 + \tau - \bar{s}\right) =  e^{\tau - \bar{s}}o\left(e^{(1 - m/2)\bar{s}}\right),$$
$$|J_4|\leq Ce^{-\theta e^{(1 - 2/m)\bar{s}}}\quad \text{for some $\theta > 0$}.$$
Putting together the above estimates yields the desired result. This ends the proof of Lemma \ref{lemm:gt12}.
\end{proof}

\noindent Applying now Lemma \ref{lem:Gro} and Lemma \ref{lemm:gt12}, we obtain for $s$ large enough,
$$g(\tau) \leq e^{\tau - \bar{s}}o\left( e^{-(m/2 - 1)\bar{s}} \right), \quad \forall \tau \in [\bar{s},s].$$
Since $\bar{s} = \frac{2s}{m}$, if we set $\tau = s$, then $r = K_0e^{\left(\frac{1}{2} - \frac{1}{m}\right)s}$ and 
$$g(s) \equiv L^{2,r(s)}_\rho (Q(s)) = o(1) \quad \text{as} \quad s \to +\infty.$$
By the regularizing effects of the semigroup $S_{\mathcal{L}}$, we then obtain
$$\sup_{|y| \leq \frac{K_0}{2}e^{\left(\frac{1}{2} -\frac{1}{m}\right)s}} Q(y,s) \leq C'(C_*,K_0, M_0)L^{2,r(s)}_\rho(Q(s)) \to 0,\; \text{as} \; s \to +\infty,$$
From \eqref{equ:defivarphi}, we see that for all $|y| \leq \frac{K_0}{2}e^{\left(\frac{1}{2} -\frac{1}{m}\right)s}$, 
$$\left|w(y,s) - \frac{\phi(s)}{\kappa}G\left(ye^{-\left(\frac{1}{2} - \frac{1}{m} \right)s}\right)\right| \leq Q(y,s) + Ce^{-\left(1 - \frac{2}{m}\right)s},$$
Noticing from Lemma \ref{ap:lemmA3} that $\frac{\phi(s)}{\kappa} = 1 + \mathcal{O}(s^{-a})$ as $s \to + \infty$, we obtain 
$$\sup_{|y| \leq \frac{K_0}{2}e^{\left(\frac{1}{2} -\frac{1}{m}\right)s}}\left|w(y,s) - G\left(ye^{-\left(\frac{1}{2} - \frac{1}{m} \right)s}\right)\right| = o(1), \quad \text{as} \quad s \to +\infty. $$
It remains to show that $m$ is even. Indeed, from \eqref{equ:asypsim}, we can see that if $m$ is not even, there would exist $\xi_0\in \mathbb{R}^n$ such that $w\left(\xi_0 e^{\left(\frac{1}{2} - \frac{1}{m}\right)s},s\right) \to \psi_m(\xi_0) \to +\infty$ as $s \to +\infty$, which contradicts the fact that $w$ is bounded as stated  in \eqref{equ:boundw}. Therefore, $m$ must be even. This concludes the proof of Proposition \ref{prop:2} and Theorem \ref{theo:pro} too.
\end{proof}

\appendix
\renewcommand*{\thesection}{\Alph{section}}
\counterwithin{theo}{section}
\section{Appendix A}\label{app:A}
The following lemma shows the asymptotic behavior of the solution of the associated ODE of equation \ref{equ:problem}:
\begin{lemm}\label{ap:lemmA1}
Let $v$ be a blow-up solution of the following ordinary differential equation:
\begin{equation}\label{ap:ODEv}
v'(t) = v^p(t) + h(v), \quad v(T) = +\infty \quad \text{for some $T > 0$},
\end{equation}
where $h$ is defined in \eqref{equ:h}. Then $v$ satisfies 
$$v(t) \sim \kappa (T - t)^{-\frac{1}{p-1}} \quad \text{as $t \to T$, where $\kappa = (p-1)^{-\frac{1}{p-1}}$}.$$
\end{lemm}
\begin{proof}
Divide \eqref{ap:ODEv} by $v^p$ and note that $\frac{h(v)}{v^p} \to 0$ as $v \to +\infty$, we see that for all $\varepsilon > 0$, there exists a number $\delta = \delta(\epsilon) > 0$ such that
\begin{equation}\label{equ:odeV}
\left | \frac{v'}{v^p} -1 \right| \leq \varepsilon, \quad \forall t \in [T-\delta, T). \end{equation}
Solving \eqref{equ:odeV} with noting that $v(T) = +\infty$ yields
\begin{align*}
(1+\varepsilon)^{-\frac{1}{p-1}}\kappa(T-t)^{-\frac{1}{p-1}} &\leq v(t) \leq (1-\varepsilon)^{-\frac{1}{p-1}}\kappa(T-t)^{-\frac{1}{p-1}}, \quad \forall t \in [T-\delta, T).
\end{align*}
This concludes the proof of Lemma \ref{ap:lemmA1}.
\end{proof}
\noindent The following lemma gives us an estimation of the perturbation term in equation \eqref{equ:divw1}:
\begin{lemm} \label{ap:lemmA2} Let $h$ be the function defined in \eqref{equ:h}, then it holds that
\begin{equation*}
j =0, 1,\quad e^{-\frac{(p-j)s}{p-1}}\left|h^{(j)}\left(e^\frac{s}{p-1}w\right) \right| \leq  Cs^{-a}\left(|w|^{p-j} + 1\right), \quad \forall s \geq \hat{s},
\end{equation*}
where $C = C(a,p,\mu,M) > 0$ and $\hat{s} = \hat{s}(a,p) > 0$ such that $\frac{\log s}{s} \leq \frac{p}{a(p-1)}$ for all $s \geq \hat{s}$.
\end{lemm}
\begin{proof}
We have from \eqref{equ:h} that for $j = 0, 1,$
$$e^{-\frac{(p-j)s}{p-1}}\left|h^{(j)}\left(e^\frac{s}{p-1}w \right) \right| \leq C'(M,\mu)\left( \frac{|w|^{p-j}}{\log^a\left(2 + e^\frac{2s}{p-1}w^2\right)} + e^{-\frac{(p - j)s}{p-1}}\right).$$
Considering the first case $w^2e^{\frac{s}{p-1}} \geq 4$, we have
$$\frac{|w|^{p-j}}{\log^a\left(2 + e^{\frac{2s}{p-1}}w^2\right)} \leq \frac{|w|^{p-j}}{\log^a \left(4e^{\frac{s}{p-1}}\right)} \leq \frac{(p-1)^a}{s^a}|w|^{p-j}.$$
Now, considering the second case $w^2e^{\frac{s}{p-1}} \leq 4$, we have $|w|^{p-j} \leq 2^{p-j}e^{-\frac{(p-j)s}{2(p-1)}}$ which yields 
$$\frac{|w|^{p-j}}{\log^a\left(2 + e^{\frac{2s}{p-1}}w^2\right)} \leq \frac{|w|^{p-j}}{\log^a(2)} \leq \frac{2^{p-j}}{\log^a2}e^{-\frac{(p-j)s}{2(p-1)}}.$$
Taking $C = \max\left\{C', \frac{2^{p}}{\log^a2}, (p-1)^a\right\}$ and $\hat{s} > 0$ such that $e^{-\frac{(p-j)s}{p-1}} \leq s^{-a}$ for all $s \geq \hat{s}$, we have the conclusion. This ends the proof of Lemma \ref{ap:lemmA2}.
\end{proof}

\noindent The following lemma shows us the existence of solutions of the associated ODE of equation \eqref{equ:divw1}:
\begin{lemm} \label{ap:lemmA3} Let $\phi$ be a positive solution of the following ordinary differential equation:
\begin{equation}\label{ap:phiODE}
\phi_s = -\frac{\phi}{p-1} + \phi^p + e^{-\frac{ps}{p-1}}h\left(e^\frac{s}{p-1} \phi\right).
\end{equation}
Then $\phi(s) \to \kappa$ as $s \to +\infty$ and $\phi(s)$ is given by
\begin{equation}\label{ap:solphi}
\phi(s) = \kappa(1 + \eta_a(s))^{-\frac{1}{p-1}}, \quad \text{where} \quad \eta_a(s) = \mathcal{O}\left(\frac{1}{s^a}\right).
\end{equation}
If $h(x) = \mu\frac{|x|^{p-1}x}{\log^a(2 + x^2)}$, we have 
$$\eta_a(s) \sim C_0\int_s^{+\infty}\frac{e^{s-\tau}}{\tau^a}d\tau = \frac{C_0}{s^a}\left(1 + \sum_{j\geq 1}\frac{b_j}{s^j}\right),$$
where $C_0 = \mu\left(\frac{p-1}{2}\right)^a$ and $b_j = (-1)^j\prod_{i = 0}^{j-1}(a+i)$.
\end{lemm}
\begin{proof}
By the following transformation 
$$ v(t) = (T-t)^{-\frac{1}{p-1}}\phi(s),\quad s = -\log(T-t),$$
equation \eqref{ap:ODEv}  is transformed into \eqref{ap:phiODE}. From Lemma \ref{ap:lemmA1}, we see that $\phi(s) \to \kappa$ as $s \to +\infty$.\\
By dividing equation \eqref{ap:phiODE} by $\phi^p$, we find that
\begin{equation}\label{equ:apODEdphi}
\left(\frac{1}{\phi^{p-1}} \right)' = \frac{1}{\phi^{p-1}} -(p-1)(1 + g(s)), \quad g(s) = \frac{1}{\phi^p}e^{-\frac{ps}{p-1}}h\left(e^\frac{s}{p-1}\phi\right).
\end{equation}
Since $\phi(s) \to \kappa$ as $s \to +\infty$, we have from Lemma \ref{ap:lemmA2} that $g(s) = \mathcal{O}\left(\frac{1}{s^a}\right)$ as $s \to +\infty$. Solving equation \eqref{equ:apODEdphi} yields 
$$\phi(s) = \kappa\left(1 + \eta_a(s)\right)^{-\frac{1}{p-1}}, \quad \text{where}\quad \eta_a(s) = \int_s^{+\infty}e^{s - \tau}g(\tau)d\tau.$$
By integration by part, we find that
\begin{equation}\label{equ:intbypar}
\int_s^{+\infty}e^{s - \tau}\tau^{-a}d\tau = \frac{1}{s^a}\left(1 + \sum_{j\geq 1}\frac{b_j}{s^j}\right),\quad b_j = (-1)^j\prod_{i = 0}^{j-1}(a+i).
\end{equation}
This follows that $\eta_a(s) = \mathcal{O}(s^{-a})$ as $s \to +\infty$. In particular case $h(x) = \mu\frac{|x|^{p-1}x}{\log^a(2 + x^2)}$, we have $g(s) = \mu\log^{-a}\left(2 + e^\frac{2s}{p-1} \phi^2(s)\right) \sim \mu\left(\frac{p-1}{2s}\right)^a$ as $s \to +\infty$. By \eqref{equ:intbypar}, we get the desired result and finish the proof of Lemma \ref{ap:lemmA3}.
\end{proof}

\section{Appendix B}
\subsection{Proof of Proposition \ref{prop:regpar}}\label{ap:proRepa}
\noindent We give the proof of Proposition \ref{prop:regpar} here. 
\begin{proof} The idea of the proof is given in Lady{\v{z}}enskaja and al. \cite{LSU68}. Note that we still get interior regularity even if we know nothing  about the initial or boundary data. Indeed, let $\tau\in (0,1)$ and fix $t_0$ such that $t_0 -\tau > 0$, we denote
$Q_{\tau}(t_0) = \mathbf{B}_{R/2}\times (t_0 - \tau, t_0) \subset Q_R$, and let $\varphi(x,t)$ be a smooth function defined in $Q_R$ such that $0 \leq \varphi(x,t) \leq 1$ and $\varphi(x,t) = 0$  for all $(x,t) \in Q_R \setminus Q_{\tau}(t_0)$. Let $k \geq 1$, define 
$$v_k(x,t) = \max\{v(x,t) - k,0\} \quad \text{and} \quad A_k(t) = \{x \in \mathbf{B}_R: v(x,t) > k\}.$$
Then, multiplying equation \eqref{equ:vinger} by $v_k\varphi^2$ and integrating over $Q_\tau(t_0)$, we find that
\begin{align*}
&\frac{1}{2}\int_{\mathbf{B}_R}v_k^2 \varphi^2dx \vert_{t_0 - \tau}^{t_0} + \int_{t_0 -\tau}^{t_0}\int_{\mathbf{B}_R}\left|\nabla v_k\right|^2\varphi^2dxdt\\
& = - \int_{t_0 - \tau}^{t_0}\int_{\mathbf{B}_R}v_k^2 \varphi \varphi_t dx dt + 2 \int_{t_0 -\tau}^{t_0}\int_{\mathbf{B}_R}\left(\nabla v_k \cdot \nabla \varphi \right) v_k\varphi dxdt\\
& \quad - \int_{t_0 -\tau}^{t_0}\int_{\mathbf{B}_R}\left(b\cdot \nabla v_k\right) v_k\varphi^2 dxdt + \int_{t_0 -\tau}^{t_0}\int_{A_k(t)}F v_k\varphi^2 dxdt.
\end{align*}
Using the assumption $|F| \leq g (|v| + 1)$ and some elementary inequalities with noticing that $\varphi(\cdot,t_0 - \tau) = 0$, we then obtain
 \begin{align}\max_{t_0 -\tau \leq t\leq t_0}\left\|v_k(t)\varphi(t)\right\|^2_{L^2(\mathbf{B}_R)}& + \int_{t_0 -\tau}^{t_0}\int_{\mathbf{B}_R}\left|\nabla v_k\right|^2\varphi^2dxdt \nonumber\\
& \leq 2\int_{t_0 -\tau}^{t_0}\int_{\mathbf{B}_R} \left(4|\nabla \varphi| + \varphi|\varphi_t| \right)v_k^2 dx dt \nonumber\\
& + 2\int_{t_0 -\tau}^{t_0}\int_{A_k(t)} (\mu_1^2 + 2g)\left(v_k^2 + k^2\right)\varphi^2 dx dt.\label{eq:tmVkreg} 
\end{align}
For the last term in the right-hand side (denote by $I$), we use Holder's inequality and \eqref{equ:conLemint}, which reads
\begin{align*}
|I| &\leq \left(\int_{t_0 -\tau}^{t_0} \|2\mu_1^2 + 4g\|_{L^{\alpha'}(A_k(t))}^{\beta'} dt \right)^\frac{1}{\beta'} \left(\int_{t_0 -\tau}^{t_0} \left\|\left(v_k^2 + k^2\right)\varphi^2\right\|_{L^{\alpha_1}(A_k(t))}^{\beta_1} dt \right)^\frac{1}{\beta_1},\\
& \leq \gamma \left(\int_{t_0 -\tau}^{t_0} \left\|\left(v_k^2 + k^2\right)\varphi^2\right\|_{L^{\alpha_1}(A_k(t))}^{\beta_1} dt\right)^\frac{1}{\beta_1} = \gamma II,
\end{align*}
where $\gamma = \gamma(\mu_1, \mu_2, R, \alpha', \beta') > 0$, $\alpha_1 = \frac{\alpha'}{\alpha' - 1}$ and $\beta_1 = \frac{\beta'}{\beta' - 1}$.\\
From pages 184 and 185 in \cite{LSU68}, we have the following interpolation identity:
$$II \leq\beta \theta_k^{\frac{2\epsilon}{r}}\left(\max_{t_0 -\tau \leq t\leq t_0}\left\|v_k(t)\varphi(t)\right\|^2_{L^2(A_k(t))} + \int_{t_0 -\tau}^{t_0}\int_{A_k(t)}\left|\nabla v_k\right|^2 \varphi^2dxdt \right) +  k^2 \sigma_k^{\frac{2(1 + \epsilon)}{r}},$$
where $\epsilon \in (0,1)$, $r \geq 2$, $\beta > 0$ are constants,
$$\theta_k = \int_{t_0 -\tau}^{t_0}\left|A_k(t)\right|^{\beta_1/\alpha_1}dt, \quad \sigma_k = \int_{t_0 -\tau}^{t_0} \|\varphi(t)\|^{\beta_1}_{L^{\alpha_1}(A_k(t))} dt.$$
Since $\theta_k \leq \tau R^{\beta_1/\alpha_1}$, we can take $\tau$ small enough such that 
$$\gamma\beta (\tau R^{\beta_1/\alpha_1})^{\frac{2\epsilon}{r}} \leq \frac{1}{2}.$$
Then from \eqref{eq:tmVkreg}, we have 
\begin{align}
\max_{t_0 -\tau \leq t\leq t_0}\left\|v_k(t)\varphi(t)\right\|^2_{L^2(\mathbf{B}_R)}& + \int_{t_0 -\tau}^{t_0}\int_{\mathbf{B}_{R}}\left|\nabla v_k\right|^2 \varphi^2 dxdt \nonumber\\
& \leq  \gamma'\left[\int_{t_0 - \tau}^{t_0} \int_{\mathbf{B}_R} \left(|\nabla \varphi| + \varphi|\varphi_t|\right) v_k^2 dxdt + k^2 \sigma_k^{\frac{2(1 + \epsilon)}{r}}\right]. \label{equ:tempIdreg}
\end{align}
By Remark 6.4, page 109 and Theorem 6.2, page 103 in \cite{LSU68}, we know that if $v$ satisfies \eqref{equ:tempIdreg} for any $k \geq 1$, then for all $(x,t) \in \mathbf{B}_{R/4}\times(t_0 - \tau/2,t_0)$,
\begin{align}
 |v(x,t)| \leq \gamma''\left[\left(\frac{R}{2}\right)^{-\frac{n+2}{2}}\left(1 + \frac{R}{2\sqrt{\tau}}\right)\left(\int_{t_0-\tau}^{t_0}\|v(t)\|^2_{L^2 (\mathbf{B}_R)}dt \right)^{1/2} \right. &\nonumber\\ 
  + \left. \left(1 + \frac{4\tau}{R}\right)^{\frac{1 + \epsilon}{r}} \right] < +\infty&.\label{eq:rereg}
\end{align}
Analogous arguments with the function $-v$ would yield the same estimate. Since $\mu_1$, $\mu_2$ and $\mu_3$ are uniformly bounded in $t_0$, this implies that estimate \eqref{eq:rereg} holds for all $(x,t) \in \mathbf{B}_{R/4}\times(\tau/2, +\infty)$. This concludes the proof of Proposition \ref{prop:regpar}.
\end{proof}
\subsection{Proof of Proposition \ref{prop:upElc}}\label{ap:upELc}
We prove Proposition \ref{prop:upElc} here. Let us first derive the upper bound for $\mathcal{E}_\psi$.
\begin{proof}[\textbf{Proof of the upper bound for $\mathcal{E}_\psi$}] Multiplying equation \eqref{equ:divw1} with $\psi^2 w_s$ and integrating on $\mathbb{R}^n$ yield 
\begin{align*}
\int_{\mathbb{R}^n}\psi^2 w_s^2 \rho dy & = -\frac{1}{2}\frac{d}{ds}\int_{\mathbb{R}^n}\psi^2 |\nabla w|^2\rho dy - 2\int_{\mathbb{R}^n}\psi w_s \nabla \psi. \nabla w \rho dy\\
&\quad  - \frac{1}{2(p-1)}\frac{d}{ds}\int_{\mathbb{R}^n}\psi^2|w|^2\rho dy + \frac{1}{p+1}\frac{d}{ds}\int_{\mathbb{R}^n}\psi^2|w|^{p+1}\rho dy\\
& + e^{-\frac{p+1}{p-1}} \frac{d}{ds}\int_{\mathbb{R}^n}\psi^2 H\left(e^{\frac{s}{p-1}}w \right) \rho dy\\
& - \frac{1}{p-1}e^{-\frac{p}{p-1}s}\int_{\mathbb{R}^n}\psi^2h\left(e^{\frac{s}{p-1}}w \right)w \rho dy.
\end{align*}
We derive the following identity from the definition \eqref{equ:localE} of the local functional $\mathcal{E}_\psi$, 
\begin{align}
\frac{d}{ds}\mathcal{E}_\psi[w](s) &= -\int_{\mathbb{R}^n} \psi^2|w_s|^2\rho dy  - 2\int_{\mathbb{R}^n}\psi w_s \nabla \psi. \nabla w \rho dy \nonumber\\
&\quad + \frac{p+1}{p-1}e^{-\frac{p + 1}{p-1}s}\int_{\mathbb{R}^n} \psi^2 H\left(e^{\frac{1}{p-1}s}w\right)\rho dy \nonumber\\
& - \frac{1}{p-1} e^{-\frac{p}{p-1}s}\int_{\mathbb{R}^n}\psi^2 h\left(e^{\frac{1}{p-1}s}w\right)w\rho dy. \label{equ:Id1lc}
\end{align}
Using the fact that $2ab \leq \frac{a^2}{2} + 2b^2$, we obtain 
$$2\psi w_s \nabla \psi. \nabla w  \leq \frac{1}{2}\psi^2 w_s^2 + 2 |\nabla \psi|^2 |\nabla w|^2.$$
Combining with \eqref{equ:estimateHhe}, we get an estimate for \eqref{equ:Id1lc} as follows:
\begin{align*}
\frac{d}{ds}\mathcal{E}_\psi[w](s) &\leq -\frac{1}{2}\int_{\mathbb{R}^n} \psi^2|w_s|^2\rho dy + 2\|\nabla \psi\|^2_{L^\infty} \int_{\mathbb{R}^n}|\nabla w|^2 \rho dy \\
& \quad + Cs^{-a}\int_{\mathbb{R}^n} |w|^{p+1}\rho dy + Cs^{-a}, 
\end{align*}
where $C = C(a, p, n, M, \|\psi\|^2_{L^\infty})$.\\
Using $(iii)$ and $(iv)$ of Proposition \ref{prop:boundEpsi}, we see that
\begin{equation}\label{equ:con1Elc}
\frac{d}{ds}\mathcal{E}_\psi[w](s) \leq C_1 \left(1 + \|w_s\|_{L^2_\rho(\mathbb{R}^n)} \right), \quad \forall s \geq \tilde{s}_3,
\end{equation}
where $C_1 = C_1\left(a, p, n, N, J_3, J_4, \|\psi\|^2_{L^\infty}, \|\nabla \psi\|^2_{L^\infty} \right)$ and $J_i$ is introduced in Proposition \ref{prop:boundEpsi}.\\
From the definition of $\mathcal{E}_\psi$ given in \eqref{equ:localE}, we have 
\begin{align*}
\mathcal{E}_\psi[w](s) & \leq \|\psi\|^2_{L^\infty} \left\{\frac{1}{2} \int_{\mathbb{R}^n}\left(|\nabla w|^2 + \frac{1}{p-1}|w|^2 \right)\rho dy
- e^{-\frac{p+1}{p-1}} \int_{\mathbb{R}^n}H\left(e^{\frac{s}{p-1}}w \right) \rho dy. \right\} \\
& = \|\psi\|^2_{L^\infty} \left\{\mathcal{E}[w](s) +  \frac{1}{p+1}\int_{\mathbb{R}^n} |w|^{p+1}\rho dy \right\}\\
& \leq \|\psi\|^2_{L^\infty} \left\{J_0 +  \frac{1}{p+1}\int_{\mathbb{R}^n} |w|^{p+1}\rho dy \right\}. \quad \forall s \geq \tilde{s}_3.
\end{align*}
Integrating on $[s, s+1]$, we obtain 
\begin{align*}
\int_s^{s+1} \mathcal{E}_\psi[w](\tau) d\tau & \leq \|\psi\|^2_{L^\infty} \left\{J_0 +  \frac{1}{p+1}\int_s^{s+1} \int_{\mathbb{R}^n} |w|^{p+1}\rho dy d\tau\right\}\\
& \leq \|\psi\|^2_{L^\infty} \left\{J_0 +  \frac{1}{p+1}\left[\int_s^{s+1} \left(\int_{\mathbb{R}^n} |w|^{p+1}\rho dy \right)^2d\tau \right]^{\frac{1}{2}}\right\}\\
& \leq C_2\left(\|\psi\|^2_{L^\infty}, J_0, J_5\right) \quad \text{(use $(v)$ of Proposition \ref{prop:boundEpsi})}.
\end{align*}
Hence,
\begin{equation}\label{equ:con2Elc}
\int_s^{s+1} \mathcal{E}_\psi[w](\tau) d\tau \leq  C_2, \quad \forall s \geq \tilde{s}_3.
\end{equation}
Thus, there exists $\tau(s) \in [s, s+1]$ such that 
$$\mathcal{E}_\psi[w](\tau(s)) = \int_s^{s+1} \mathcal{E}_\psi[w](\tau')d\tau' \leq C_2.$$ 
We then have 
\begin{align*}
\mathcal{E}_\psi[w](s) & = \mathcal{E}_\psi[w](\tau(s)) + \int_{\tau(s)}^{s} \frac{d}{ds}\mathcal{E}_\psi[w](\tau')d\tau'\\
& \leq C_2 + \int_s^{s+1} C_1 \left(1 + \|w_s\|_{L^2_\rho(\mathbb{R}^n)} \right)d\tau' \leq C_2'. \;\; \text{(use $(i)$ of Proposition \ref{prop:boundEpsi})}
\end{align*}
This concludes the proof of the upper bound for $\mathcal{E}_\psi$.\\

\noindent It remains to prove the lower bound in order to conclude the proof of Proposition \eqref{prop:boundEpsi}.
\end{proof}
\begin{proof}[\textbf{Proof of the lower bound for $\mathcal{E}_\psi$}] Multiplying equation \eqref{equ:divw1} with $\psi^2 w$ and integrating on $\mathbb{R}^n$ yield 
\begin{align}
\frac{1}{2}\frac{d}{ds}\int_{\mathbb{R}^n} (\psi w)^2\rho dy & = -2\mathcal{E}_\psi[w](s) + \frac{p+1}{p-1}\int_{\mathbb{R}^n} \psi^2 |w|^{p+1}\rho dy \nonumber\\
& \quad - 2 \int_{\mathbb{R}^n} \psi w \nabla \psi . \nabla w \rho dy \nonumber\\
& \quad - 2 e^{-\frac{p+1}{p-1}s}\int_{\mathbb{R}^n} \psi^2 H\left(e^\frac{s}{p-1}w\right)\rho dy \nonumber\\
& \quad + e^{-\frac{p}{p-1}s}\int_{\mathbb{R}^n} \psi^2 h\left(e^\frac{s}{p-1}w\right)w\rho dy. \label{equ:Id2lcW}
\end{align}
We now control the new term $J_\psi[w](s) = 2\int_{\mathbb{R}^n} \psi w \nabla \psi . \nabla w \rho dy$ as follows: 
\begin{align}
J_\psi[w](s) & = -2 \int_{\mathbb{R}^n} w \nabla. \left(\psi w \nabla \psi\rho \right) dy \nonumber\\
& = - 2 \int_{\mathbb{R}^n} |w|^2 |\nabla \psi|^2 \rho dy - 2\int_{\mathbb{R}^n} \psi w \nabla \psi . \nabla w \rho dy \nonumber\\
& \quad - 2 \int_{\mathbb{R}^n} \psi |w|^2 \Delta \psi \rho dy + \int_{\mathbb{R}^n} \psi |w|^2y.\nabla \psi \rho dy.\nonumber\\ 
& = - \int_{\mathbb{R}^n} |w|^2 |\nabla \psi|^2 \rho dy -  \int_{\mathbb{R}^n} \psi |w|^2 \Delta \psi \rho dy + \frac{1}{2}\int_{\mathbb{R}^n} \psi |w|^2y.\nabla \psi \rho dy.\nonumber\\ 
& \leq \left[\|\psi\|_{L^\infty} \left(\|\Delta \psi\|_{L^\infty}  + \frac{1}{2}\|y.\nabla \psi\|_{L^\infty}\right) \right] \int_{\mathbb{R}^n} |w|^2\rho dy\nonumber\\
&\leq J_2 C_1(\psi), \quad \forall s \geq \tilde{s}_3  \quad \text{(use $(ii)$ of Proposition \ref{prop:boundEpsi})}. \label{equ:estJws}
\end{align}
Using \eqref{equ:estimateHhe} and \eqref{equ:Id2lcW},  we obtain 
\begin{align*}
\frac{1}{2}\frac{d}{ds}\int_{\mathbb{R}^n} (\psi w)^2\rho dy &\geq - 2 \mathcal{E}_\psi - J_2C_1(\psi) - C_2 s^{-a}\\
&\qquad \qquad  + \left(\frac{p+1}{p-1} - C_2s^{-a}\right)\int_{\mathbb{R}^n} \psi^2 |w|^{p+1}\rho dy.
\end{align*}
Taking $S$ large enough such that $C_2s^{-a} \leq \frac{p+1}{2(p-1)}$ for all $s \geq S$, we obtain for all $s \geq \max\{S, \tilde{s}_3\}$,
\begin{equation*}
\frac{1}{2}\frac{d}{ds}\int_{\mathbb{R}^n} (\psi w)^2\rho dy \geq -(2\mathcal{E}_\psi + C_3) + \frac{p+1}{2(p-1)}\int_{\mathbb{R}^n} \psi^2 |w|^{p+1}\rho dy,
\end{equation*}
where $C_3 = J_2C_1 + \frac{p+1}{2(p-1)}$.\\

\noindent Let $g(s) = 2 \mathcal{E}_\psi + C_3$ and $f(s) = \frac{1}{2}\int_{\mathbb{R}^n} \psi^2 |w|^{2}\rho dy$. Using Jensen's inequality, we have
\begin{align*}
f(s)^\frac{p+1}{2} &= 2^{-\frac{p+1}{2}}\left(\int_{\mathbb{R}^n} \psi^2 |w|^{2}\rho dy\right)^{\frac{p+1}{2}}\\
&\leq 2^{-\frac{p+1}{2}}\int_{\mathbb{R}^n} (\psi |w|)^{p+1}\rho dy \leq 2^{-\frac{p+1}{2}} \|\psi\|_{L^\infty}^{p-1}\int_{\mathbb{R}^n} \psi^2 |w|^{p+1}\rho dy.
\end{align*}
We therefore obtain for all $s \geq S_1 = \max\{S,\tilde{s}_3\} $,
\begin{equation}\label{equ:fff}
f'(s) \geq -g(s) + C_4f(s)^{\frac{p+1}{2}}.
\end{equation}
From \eqref{equ:con1Elc}, we also have
\begin{equation}\label{equ:ggg}
g'(s) \leq C_5 + h(s), \quad \forall s \geq \tilde{s}_3,
\end{equation}
where $h(s) = C_5\|w_s\|_{L^2_\rho(\mathbb{R}^n)}$ and $m = \int_{\tilde{s}_3}^{+\infty} h(s)ds \leq C_6$ by using $(i)$ of Proposition \ref{prop:boundEpsi}, where $C_5, C_6$ are some positive constants.\\
We claim that the function of $g$ is bounded from below by some constant $M$. Arguing by contradiction, we suppose that there exists a time $s^* \geq S_1$ such that $g(s^*) \leq - M$. Then for all $s \geq s^*$, we write 
\begin{align*}
g(s) & = g(s^*) + \int_{s^*}^s g'(\tau)d\tau \leq -M + \int_{s^*}^s (C_5+ h(\tau))d\tau\\
& \leq -M + m + C_5(s - s^*).
\end{align*}
Thus, we have by \eqref{equ:fff},
$$f' \geq M - m - C_5(s - s^*) + C_4f^{\frac{p+1}{2}}, \quad f(s^*) \geq 0.$$
On the other hand, we know that the solution of the following equation 
$$f' \geq 1 + C_5f^{\frac{p+1}{2}}, \quad f(s^*) \geq 0$$
blows up in finite time before 
$$s = s^* + \int_{0}^{+\infty}\frac{d\xi}{1 + C_4f^\frac{p+1}{2}} = s^* + T^*.$$
On the interval $[s^*, s^* + T^*]$, we have 
$$M -m - C_5(s - s^*) \geq M - m - C_5T^*.$$
Thus, we fix $M = m + C_5T^* + 1$ to get $M -m - C_5(s-s^*) \geq 1$ for all $s \in [s^*, s^* + T^*]$. \\
Therefore, $f$ blows up in some finite time before $s^* + T^*$. But this contradicts with the existence global of $w$. This follows \eqref{equ:upElc} and we complete the proof of Proposition \ref{prop:upElc}.
\end{proof}

\section{Appendix C}
\noindent We claim the following:
\begin{lemm}[\textbf{Estimate on $\bar{F}$}]\label{ap:lemmbounF_H} For $s$ large enough, we have
$$\left|\bar{F}(V,s) - \frac{p}{2\kappa}V^2\right| \leq C|V|^3 + \frac{C|V|^2}{s^{a - 1}},$$
where $C = C(a,p,M,\mu) > 0$.
\end{lemm}
\begin{proof} Consider the Taylor expansion of the nonlinear terms $F$ and $H$, we have 
\begin{align*}
F(v) &= \frac 12 p(p-1)\phi^{p-2} v^2 + \gamma_1 v^3, \quad H(v,s) = \gamma_2v^2,
\end{align*}
where 
$$\gamma_1 = \frac{1}{6}p(p-1)(p-2)|\phi + \theta_1v|^{p-3}, \quad \gamma_2 = \frac 12 e^{-\frac{(p-2)s}{p-1}}h''\left(e^\frac{s}{p-1}(\phi + \theta_2v)\right),$$
with $\theta_i \in [0,1]$, $i = 1,2$.\\
We claim the following: for $s$ large,
\begin{equation}\label{equ:apTmEstphithe}
|\gamma_1| \leq C \quad \text{and} \quad |\gamma_2| \leq \frac{C}{s^a}.
\end{equation}
Let us leave the proof of \eqref{equ:apTmEstphithe} later and continue the proof of Lemma \ref{ap:lemmbounF_H}. Recalling from Lemma \ref{ap:lemmA3} that $\phi(s) = \kappa + \mathcal{O}(s^{-a})$ as $s \to +\infty$, we derive 
$$\left|F(v) + H(v,s) - \frac{p}{2\kappa}v^2\right| = \mathcal{O}\left(\frac{|v|^2}{s^a}\right) + \mathcal{O}(|v|^3), \quad \text{as} \; s \to +\infty.$$
From the definition of $\bar{F}$ and the fact that $\beta(s) = 1 + \mathcal{O}(\frac{1}{s^{a-1}})$ as $s \to +\infty$, we have for $s$ large enough,
\begin{align*}
\left|\bar{F}(V,s) - \frac{p}{2\kappa}V^2\right| &= \left|\beta(s)\left( F(v) + H(v,s)\right) - \frac{p}{2\kappa}v^2 \beta^2 \right|\\
&\leq \left|F(v) + H(v,s) - \frac{p}{2\kappa}v^2\right| + \frac{C|v|^2}{s^{a-1}}\\
&\leq \frac{C|v|^2}{s^a} + C|v|^3 + \frac{C|v|^2}{s^{a-1}} \leq C|V|^3 + \frac{C|V|^2}{s^{a-1}},
\end{align*}
which concludes the proof of Lemma \ref{ap:lemmbounF_H}, assuming that \eqref{equ:apTmEstphithe} holds.\\

\noindent Let us now give the proof of \eqref{equ:apTmEstphithe}. Since $\phi(s) \to \kappa$ as $s \to +\infty$, we can take $s_* > 0$ such that 
$$\frac{3\kappa}{4} \leq \phi(s) \leq \frac{5\kappa}{4}, \quad \forall s \geq s_*.$$
Let us bound $|\gamma_1|$. If $p \geq 3$, by the boundedness of $|\phi|$ and $|v|$, then $|\gamma_1|$ is already bounded. If $p \in (1,3)$, we consider the case $|\theta_1 v| \leq \frac{\kappa}{2}$, then the case $|\theta_1 v| > \frac{\kappa}{2}$.  In the first case, we have $|\phi + \theta_1v| \geq \frac{\kappa}{4}$ for all $s \geq s_*$, then $|\gamma_1| \leq C|\phi + \theta_1v|^{p-3} \leq C\left(\frac{\kappa}{4}\right)^{p-3}$ for all $s \geq s_*$ . Now, considering the second case where $|\theta_1v| > \frac{\kappa}{2}$, note that in this case, we have $\theta_1 \ne 0$ and $\phi < \frac{5}{2}|\theta_1 v|$ for all $s \geq s_*$. From the definition of $F(v)$, we have
\begin{align*}
|\gamma_1 v^3| &= \left||\phi + v|^{p-1}(\phi + v) - \phi^p - p\phi^{p-1}v - \frac 12 p(p-1)\phi^{p-2} v^2\right|\\
& \leq C(|v|^p + v^2), \quad \forall s \geq s_*.
\end{align*}
This yields $|\gamma_1| \leq C(|v|^{p-3} + |v|^{-1}) \leq C \left((\kappa/2\theta_1)^{p-3} + (\kappa/2\theta_1)^{-1}\right)$ for all $s \geq s_*$. This concludes the proof of the first estimate of \eqref{equ:apTmEstphithe}.\\

Let us now prove that $|\gamma_2| \leq Cs^{-a}$ for $s$ large enough. From \eqref{equ:h}, we have
\begin{equation}\label{equ:apgamma2}
|\gamma_2| \leq  M\frac{|\phi + \theta_2v|^{p-2}}{\log^a(2 + e^\frac{s}{p-1}(\phi + \theta_2v)^2)}.
\end{equation}
If $p > 2$, by the same technique given in the proof of Lemma \ref{ap:lemmA2}, we can show that \eqref{equ:apgamma2} implies
\begin{equation*}
|\gamma_2| \leq \frac{C}{s^a}(|\phi + \theta_2v|^{p-2} + 1) \leq \frac{2C}{s^a},\quad \forall s \geq s'(a,p).
\end{equation*}
If $p \in (1,2]$, we consider the first case $|\theta_2 v| \leq \frac{\kappa}{2}$, which implies $|\phi(s) + \theta_2v| \geq \frac{\kappa}{4}$ for all $s \geq s_*$. From \eqref{equ:apgamma2}, we derive
$$|\gamma_2| \leq \frac{C (\kappa/4)^{p-2}}{\log^a(2 + e^\frac{s}{p-1}(\kappa/4)^2)} \leq \frac{2C}{s^a},\quad  \text{for $s$ large}.$$
In  the case where $|\theta_2 v| > \frac{\kappa}{2}$, we note that $\theta_2 \ne 0$ and $ \phi(s)\leq \frac{5}{2}|\theta_2 v|$ for all $s \geq s_*$. Using the definition of $H(v,s)$ and \eqref{equ:h}, we find that
\begin{align*}
|\gamma_2 v^2|&\leq C\left( \frac{|\phi + v|^p}{\log^a \left(2 + e^\frac{2s}{p-1}(\phi + v)^2\right)} + \frac{\phi^p}{\log^a \left(2 + e^\frac{2s}{p-1}\phi^2\right)} \right.\\
& \qquad \qquad \left.+ \frac{\phi^{p-1}v}{\log^a \left(2 + e^\frac{2s}{p-1}\phi^2\right)} + e^{-\frac{ps}{p-1}} + e^{-s}\right)\\
&\leq \frac{C}{s^a}\left(|\phi + v|^p + \phi^p + \phi^{p-1}|v| + 1\right) \leq \frac{2C}{s^a}(|v|^p + 1),
\end{align*}
for $s$ large. This yields $|\gamma_2| \leq \frac{2C}{s^a}(|v|^{p-2} + |v|^{-2}) \leq 
\frac{2C}{s^a}\left((\frac{\kappa}{2\theta_2})^{p-2} + (\frac{\kappa}{2\theta_2})^{-2}\right) \leq \frac{3C}{s^a}$. This concludes the proof of \eqref{equ:apTmEstphithe} and the proof of Lemma \ref{ap:lemmbounF_H} also.
\end{proof}
\begin{lemm}\label{lemm:elementary}
Let $\alpha(s)$ be a solution of 
\begin{equation}\label{equ:baseode}
\alpha'(s) = \alpha^2(s) + \mathcal{O}\left(\frac{1}{s^q}\right), \quad q \in (2,3],
\end{equation}
which exists for all time. Then
\begin{align}
\text{either}&\quad \alpha(s) = - \dfrac{1}{s} + \mathcal{O}\left(\dfrac{1}{s^q} \right) \quad \text{or } \quad \alpha(s) = \mathcal{O}\left(\dfrac{1}{s^q} \right), \quad \text{if}\quad q \in(2, 3),\label{eq:ApCrelq23}\\
\text{either}&\quad \alpha(s) = - \dfrac{1}{s} + \mathcal{O}\left(\dfrac{\log{s}}{s^2} \right) \quad \text{or } \quad \alpha(s) = \mathcal{O}\left(\dfrac{1}{s^2} \right), \quad \text{if}\quad q = 3. \label{eq:ApCrelq3}
\end{align}
\end{lemm}
\begin{proof} Let us first show that  
\begin{equation}\label{equ:fir2Cas}
\text{either} \;\; \alpha(s) = \mathcal{O}\left(\frac{1}{s^{1 + \sigma}}\right)
\;\; \text{or} \;\; \alpha(s) = -\frac{1}{s} + \mathcal{O}\left(\frac{1}{s^{1 + \sigma'}}\right) \;\; \text{as} \; s \to + \infty,
\end{equation}
for some $\sigma \in \left(0,\frac{q-2}{2}\right)$ and $\sigma' = q - 2 - 2\sigma$.\\

\noindent Fix $s_0$ large enough et let $\sigma \in \left(0,\frac{q-2}{2}\right)$. If 
$|\alpha(s)| \leq \frac{1}{s^{1 + \sigma}}$, for all $s \geq s_0$, then we are done. If not, namely there exists a time $s_1 > s_0$ such that $|\alpha(s_1)| > \frac{1}{s^{1 + \sigma}}$, we have two possibilities: \\
\begin{equation}\label{eq:apdC1}
|\alpha(s)| > \frac{1}{s^{1 + \sigma}}, \quad \forall s \geq s_1,
\end{equation}
or there exists a time $s_2 > s_1$ such that 
\begin{equation}\label{eq:apdC2}
|\alpha(s_2)| = \frac{1}{s_2^{1+\sigma}}\quad \text{and} \quad |\alpha(s_2)| \leq \frac{1}{s^{1 + \sigma}},\quad \forall s \in (s_2, s_2 + \delta),\;\delta > 0.
\end{equation}
If \eqref{eq:apdC1} is the case, then we have by equation \eqref{equ:baseode},
$$\left(\frac{1}{\alpha}\right)' = 1 + \mathcal{O}\left(\frac{1}{s^{q - 2 - 2\sigma}} \right), \quad \forall s \geq s_1,$$
which yields \eqref{equ:fir2Cas} by integration.\\
If \eqref{eq:apdC2} is the case, we assume that $\alpha(s_2) > 0$, then $\alpha'(s_2) \leq -\frac{1 + \sigma}{s_2^{2 + \sigma}} < 0$. By equation \eqref{equ:baseode} and note that $2+ 2\sigma < q$, we have $\alpha'(s_2) > 0$ and a contradiction follows. If $\alpha(s_2) < 0$, then $\alpha'(s_2) \geq \frac{1 + \delta}{s_2^{2 + \delta}}$, by equation \eqref{equ:baseode}, we get 
$$\frac{1 + \delta}{s_2^{2 + \delta}} \leq \alpha'(s_2) \leq \frac{1}{s^{2 + 2\sigma}_2} + \frac{1}{s_2^q}.$$
Since $2 + \delta < 2 + 2\delta < q$, we have a contradiction and \eqref{equ:fir2Cas} follows.\\

\noindent We now use \eqref{equ:fir2Cas} in order to conclude Lemma \ref{lemm:elementary}. Let us give the proof in the case $q = 3$. Assume $\alpha(s) = \mathcal{O}\left(\frac{1}{s^{1 + \sigma}}\right)$ for some $\sigma \in (0, \frac{1}{2})$, then \eqref{equ:baseode} yields
$$\alpha'(s) = \mathcal{O}\left(\frac{1}{s^{2 + 2\sigma}}\right) + \mathcal{O}\left(\frac{1}{s^3}\right) = \mathcal{O}\left(\frac{1}{s^{2 + 2\sigma}}\right).$$
By integration, we get $\alpha(s) = \mathcal{O}\left(\frac{1}{s^{1 + 2\sigma}}\right)$. Using this estimate, we obtain $\alpha'(s) = \mathcal{O}\left(\frac{1}{s^3}\right)$ and the conclusion follows.\\
Let us consider 
$$\alpha(s) = -\frac{1}{s} + \beta(s), \quad \text{with} \quad \beta(s) = \mathcal{O}\left(\frac{1}{s^{1 + \sigma'}}\right), \quad \sigma' = 1 - 2\sigma.$$
Substituting this into \eqref{equ:baseode} yields
$$\beta'(s) = \frac{2\beta(s)}{s} + \beta^2(s) + \mathcal{O}\left(\frac{1}{s^3}\right).$$
Multiplying this equation by $s^2$, we find
$$\left[s^2 \beta(s)\right]' = s^2 \beta^2 + \mathcal{O}\left(\frac{1}{s}\right) = \mathcal{O}\left(\frac{1}{s^{2\sigma'}}\right) + \mathcal{O}\left(\frac{1}{s}\right).$$
If $\sigma' \geq \frac{1}{2}$, then $\left[s^2 \beta(s)\right]' =  \mathcal{O}\left(\frac{1}{s}\right)$ which follows $\beta(s) = \mathcal{O}\left(\frac{\log s}{s^2}\right)$. If $\sigma' < \frac{1}{2}$, then $\beta(s) = \mathcal{O}\left(\frac{1}{s^{1 + 2\sigma'}}\right)$. Using this estimate and repeating the process again, we would obtain  $\beta(s) = \mathcal{O}\left(\frac{\log s}{s^2}\right)$ and \eqref{eq:ApCrelq3} then follows. Since the argument is similar in the case $q \in (2,3)$, we escape here and concludes the proof of Lemma \ref{lemm:elementary}.
\end{proof}

%===========================================
%\section*{\refname}
%\bibliographystyle{plain} 
%\bibliography{D:/PhD/Travail/mybib}
\def\cprime{$'$}

\vspace*{0.5cm}
\begin{description}
\item \textbf{Address:} Universit\'e Paris 13, Sorbonne Paris Cit\'e, Institut Galil\'ee, LAGA, 99 avenue J.B. Cl\'ement, 93430 Villetaneuse, France.
\item \textbf{E-mail:} \text{vtnguyen@math.univ-paris13.fr}
\end{description}

%\noindent \textbf{Correction}:\\
%page 1: change the hypothesis on $h$ given in \eqref{equ:h}.\\
%page 3: change the estimate \eqref{equ:estimateH} corresponding to the hypothesis on $h$.\\
%page 4: delete "$p$ satisfy \eqref{equ:rangep}" in Theorem \ref{theo:lya}.\\
%page 12: delete "such that" in Lemma \ref{lemm:cri}.\\
%page 23: add the line "(see Lemma \ref{ap:lemmbounF_H} for the proof of \eqref{rem:boundF_H}, and note that \eqref{equ:estFbar2} follows from \eqref{rem:boundF_H})."\\
%page 30, line 7: $\|V(s)\|_{L^2_\rho}^2 \leq ... => \|V(s)\|_{L^2_\rho} \leq ...$.\\
%page 30, line 8: $\|\bar{F}(V(s))\|_{L^2_\rho}^2 \leq ... => \|\bar{F}(V(s))\|_{L^2_\rho} \leq ...$.\\
%page 38: modify the proof of Lemma \ref{ap:lemmA2} corresponding to the hypothesis on $h$.\\
%page 45: Give the proof of \eqref{rem:boundF_H} in Lemma C.1.
\end{document}